\documentclass{amsart}
\usepackage[english]{babel}
  \usepackage[
    backend=biber,
    style=alphabetic,
    doi=false,
    isbn=false,
    url=false,
    eprint=false,
    giveninits=true
  ]{biblatex}
\usepackage{longtable}
\usepackage{amsfonts}
\usepackage{booktabs}
\usepackage{seqsplit}
\usepackage{array}
\usepackage{microtype}
\usepackage{listings}
\usepackage{xcolor}
\usepackage[T1]{fontenc}
\usepackage[all,arc]{xy}
\usepackage{enumerate}
\usepackage{fancyhdr}
\usepackage[utf8]{inputenc}
\usepackage{amsfonts,empheq}
\usepackage{extpfeil}
\usepackage{fullpage}
\usepackage{amssymb}
\usepackage{comment}
\usepackage{hyperref} 
\usepackage{setspace}
\hypersetup{%
  colorlinks=true,%
  linkcolor=black,%
  citecolor=black,%
  filecolor=black,%
  menucolor=black,
  urlcolor=black,%
  pdfnewwindow=true,%
  pdfstartview=FitBH}  
\usepackage{url}
\usepackage{fancyhdr}
\usepackage{graphicx}
\usepackage{pdflscape}
\usepackage{geometry}
\usepackage{amsmath}
\usepackage{enumitem}
\usepackage{mathtools}
\usepackage{color}
\usepackage{mathrsfs}
\usepackage{tikz-cd}
\usepackage{etoolbox}
\usepackage{amsthm}
\usepackage{lmodern}

\lstdefinelanguage{Macaulay2}{
  sensitive=true,
  morekeywords={and,break,catch,continue,do,else,for,from,global,if,in,list,
    local,new,not,of,or,return,step,symbol,then,throw,time,to,try,when,while,xor},
  morekeywords=[2]{ZZ,QQ,RR,CC,Ring,Ideal,Module,Matrix,List,Sequence,HashTable},
  morekeywords=[3]{ideal,matrix,gens,gb,res,ker,image,coker,dim,degree,betti,
    decompose,primaryDecomposition,hilbertSeries,rank,transpose,map,sub,apply,
    select,print,needsPackage},
  morecomment=[l]{--},          
  morecomment=[s]{-*}{*-},      
  morestring=[b]",
  moredelim=[s][\color{teal}]{///}{///},   
}

\allowdisplaybreaks

\usepackage{titlesec}
\makeatletter
\renewcommand\section{\@startsection{section}{3}{\z@}%
                                     {-3.25ex\@plus -1ex \@minus -.2ex}%
                                     {1.5ex \@plus .2ex}
                                     {\normalfont\LARGE\bfseries}}
\renewcommand\subsection{\@startsection{subsection}{3}{\z@}%
                                     {-3.25ex\@plus -1ex \@minus -.2ex}%
                                     {-1.5ex \@plus -.2ex}
                                     {\normalfont\normalsize\bfseries}}
\patchcmd{\@maketitle}
  {\ifx\@empty\@dedicatory}
  {\ifx\@empty\@date \else {\vskip3ex \centering\footnotesize\@date\par\vskip1ex}\fi
   \ifx\@empty\@dedicatory}
  {}{}
\patchcmd{\@adminfootnotes}
  {\ifx\@empty\@date\else \@footnotetext{\@setdate}\fi}
  {}{}{}
\makeatother

\newtheorem{theorem}{Theorem}[section]
\newtheorem{corollary}[theorem]{Corollary}
\newtheorem{lemma}[theorem]{Lemma}

\newtheorem{proposition}[theorem]{Proposition}

\newtheorem{remark}[theorem]{Remark}

\theoremstyle{definition}

\newtheorem{definition}[theorem]{Definition}

\numberwithin{equation}{section}

\newcommand{\sub}{\subseteq}

\newcommand{\GG}{\mathbb{G}}

\newcommand{\ZZ}{\mathbb{Z}}
\newcommand{\NN}{\mathbb{N}}

\newcommand{\QQ}{\mathbb{Q}}

\newcommand{\FF}{\mathbb{F}}

\newcommand{\kk}{\mathcal{K}}
\newcommand{\xx}{\mathcal{X}}

\newcommand{\bb}{\mathcal{B}}

\newcommand{\cc}{\mathcal{C}}

\newcommand{\oo}{\mathcal{O}}
\newcommand{\aaa}{\mathcal{A}}
\newcommand{\ch}{\mathrm{ch}}
\newcommand{\A}{\mathbb{A}}

\newcommand{\Gg}{\mathfrak{g}}
\newcommand{\gl}{\mathfrak{gl}}
\newcommand{\sll}{\mathfrak{sl}}
\newcommand{\psl}{\mathfrak{psl}}

\newcommand{\Hh}{\mathfrak{h}}
\newcommand{\Nn}{\mathfrak{n}}
\newcommand{\ol}{\overline}

\DeclareMathOperator{\rank}{rank}
\DeclareMathOperator{\ad}{ad}

\DeclareMathOperator{\Lie}{Lie}

\DeclareMathOperator{\Tr}{Tr}

\DeclareMathOperator{\Spec}{Spec}

\newcommand{\Bb}{\mathfrak{b}}

\title{The commuting scheme of $\mathfrak{g}_2$ and a degree formula for commuting schemes}
\author{Pallav Goyal}

\address{P.~Goyal: Centre for Quantum Mathematics, Department of Mathematics and Computer Science, University of Southern Denmark}
\email{pallavg@imada.sdu.dk}
\date{}

\begin{document}

\begin{abstract}
Over a field of characteristic zero, we show that the commuting scheme of the exceptional Lie algebra $\mathfrak{g}_2$ is Cohen-Macaulay and normal (in particular, reduced). We determine its Hilbert series, show that it has degree $1825$ and is not Gorenstein, and determine an explicit linear system of parameters. A crucial ingredient of (one of) our proofs is a root theoretic formula for the degree of the commuting scheme of any reductive Lie algebra whose Borel subalgebra has a complete-intersection commuting scheme, which includes all classical Lie algebras of small rank as well as $E_6$ and $G_2$. Finally, we give a criterion that implies Cohen-Macaulayness and reducedness of the commuting scheme in any given positive characteristic and use it to show that the commuting scheme is Cohen-Macaulay and reduced for every prime $5\leq p\leq1000$. We also resolve the `bad' characteristic cases and show that the commuting scheme is reduced and Cohen-Macaulay over fields of characteristic $3$, and reduced but not Cohen-Macaulay over fields of characteristic $2$ where it has degree $438$.
\end{abstract}

\maketitle

\section{Introduction}

An extensively studied object in algebraic geometry and representation theory is the commuting scheme $\cc(\Gg)$ of a Lie algebra $\Gg$, which is the closed subscheme of $\Gg\times \Gg$ whose defining equations are given by the vanishing of the Lie bracket. That is, roughly speaking, the scheme $\cc(\Gg)$ consists of pairs of elements of $\Gg$ that commute with each other. One of the most fundamental and longstanding conjectures in the field is that the commuting scheme is reduced, that is, the ideal defined by the vanishing of the Lie bracket is a radical ideal. Since the scheme is known to be generically reduced in characteristic zero and in good characteristic, proving that it is Cohen-Macaulay is enough to establish its reducedness. 

While a resolution of this conjecture in full generality resisted a lot of effort over the years, progress has been made in several directions, some of which we recall here: It was shown that the underlying variety of $\cc(\Gg)$ is irreducible for $\Gg=\gl_n$ independently in \cite{MT52} and \cite{Ger} and for general reductive $\Gg$ in \cite{Ri}. The Cohen-Macaulayness of $\cc(\Gg)$ was resolved for $\Gg=\gl_n$ when $n=3$ in \cite{Thom} and $n=4$ in \cite{Fre94} by computational methods. Note that the scheme $\cc(\Gg)$ has an adjoint action of the connected reductive group $G$ with $\Gg=\Lie(G)$. It was shown that the categorical quotient $\cc(\Gg)//G$ is reduced in several special cases in \cite{GG, Vac, Dom, L, BD23, NC24}, before being proven in general in \cite{LNY}. The approaches in \cite{GG} and \cite{L} are via the study of `almost commuting schemes' of type $A$ and $C$ respectively. (Also see the work of the author \cite{Goy24} for another proof in type $C$.) In \cite{Gin12}, it was shown that the normalization of the commuting scheme is Cohen-Macaulay by the study of an `isospectral commuting scheme'. Finally, we note that resolutions of the reducedness conjecture of the commuting scheme in various degrees of generality have appeared in the preprints \cite{SYX}, \cite{Char} and \cite{SY}.

In this paper, we do not rely on any of these results or proofs for Cohen-Macaulayness or reducedness. Rather, we are interested in studying the commuting scheme $\cc_K(\Gg_2)$ (which we will henceforth denote by $\cc_K$) of the smallest exceptional simple Lie algebra $\Gg_2$ over a field $K$ via computational methods, inspired by \cite{Fre94}, as well as a root theoretic formula for the degree of the commuting scheme, which should be of independent interest. Note that since the commuting ideal is generated by homogeneous quadratic polynomials, the coordinate ring of $\cc_K$ is naturally graded, and we fix this grading throughout the paper. For a graded ring $R$, we denote by $\mathrm{HS}(R)$ or $\mathrm{HS}(\Spec(R))$ the Hilbert series of $R$.

The main result of this paper is the following:

\begin{theorem} \label{thm:statement}
The commuting scheme $\cc_K$ is reduced, Cohen-Macaulay and normal when $K$ is an arbitrary field of characteristic $0$. Furthermore, its Hilbert series $\mathrm{HS}(\cc_K)$ is given by $H(t)(1-t)^{-16}$ where:
\[H(t):=1 + 12t + 64t^2 + 198t^3 + 388t^4 + 492t^5 + 395t^6 + 198t^7 + 64t^8 + 12t^9 + t^{10}.\]
In particular, the degree of $\cc_K$ is $H(1)=1825$.
\end{theorem}

As a corollary of the above theorem, we get the doubled Chevalley restriction isomorphism of \cite{LNY} in the special case $\Gg=\Gg_2$: The above theorem implies that $\cc_K//G_2$ is reduced. This together with density of $G_2\cdot (\Hh\times \Hh)$ in $\cc_K$ (by \cite{Ri}), makes the restriction map $K[\cc_K]^{G_2}\to K[\Hh\times \Hh]^W$ injective, where $W=D_6$ (the dihedral group with $12$ elements) is the Weyl group of $G_2$. Surjectivity of this map holds for every reductive Lie algebra by Joseph (\cite{Josn}), and thus we get the isomorphism $\cc_K//G_2 \simeq (\Hh\times \Hh)//W$.

The Hilbert series computation above also allows us to infer the following: The coefficients of the polynomial $H(t)$ form the vector $(1,12,64,198,388,492,395,198,64,12,1)$ which is not symmetric since $388\neq 395$ (although it is quite close to being symmetric!) Since $\cc_K$ is Cohen-Macaulay, this asymmetry implies that it is not Gorenstein by \cite{Stan}. Next, the Hilbert series of $\cc_K$ relative to that of $\Gg_2\times \Gg_2$ is $H(t)(1-t)^{12}=1-14t^2+2t^3+\cdots$. The term $-14t^2$ says that the commuting ideal defining $\cc_K$ is generated by $14$ quadratic polynomials (which are exactly the $14$ coefficients of the commutator $[x,y]$). Furthermore, the term $2t^3$ says that there are exactly $2$ linearly independent syzygies in degree $3$ between these $14$ elements. We find these syzygies explicitly in terms of the Killing form in Lemma~\ref{lem:killing}.

Two chief ingredients go into the proof of Theorem~\ref{thm:statement}:
\begin{enumerate}
\item An explicit linear system of parameters found using Gr\"obner basis computations,
\item A root theoretic formula for computing the degree of the commuting scheme.
\end{enumerate}

We describe the latter here. Consider an arbitrary reductive Lie algebra $\Gg$ over an algebraically closed field of characteristic $0$. Consider the following data associated with $\Gg$: A Borel subalgebra $\Bb$, the Weyl group $W$ with length function $\ell:W\to \NN_{\geq 0}$, the set $\Phi^+$ of positive roots and let $N=|\Phi^+|$. Then, we prove the following formula for $\deg(\cc(\Gg))$:

\begin{theorem}\label{thm:intro-degree-formula}
Suppose that the commuting scheme $\cc(\Bb)$ of the Borel subalgebra $\Bb$ is a complete intersection. Define
\[
P:=\prod_{\gamma\in\Phi^+}(\gamma-1)^2(\gamma+2)
\]
and let $P_N$ be its homogeneous component of degree $N$. Then
\[
\deg(\cc(\Gg))=\frac{(-1)^N}{|W|}\frac{1}{\prod_{\gamma\in\Phi^+}\gamma}\sum_{w\in W} (-1)^{\ell(w)}w(P_N).
\]
\end{theorem}
The above formula says that $\deg(\cc(\Gg))$ is precisely obtained by applying the Demazure operator associated with the longest element of the Weyl group to the polynomial $P_N$ (up to a factor of $(-1)^N/|W|$). The proof of the above theorem is geometric and uses a Grothendieck-Springer type construction to relate the commuting schemes of a reductive Lie algebra and that of its Borel subalgebra, with explicit computations done using a character formula version of the Borel-Weil-Bott theorem (see Section~\ref{sec:degree-f-la} for more details).

By the results of \cite{GG18}, the hypothesis of the above theorem holds when every simple factor of $\Gg$ is of type $A_l$ ($l\leq 15$), $B_l$ ($l\leq6$), $C_l$ ($l\leq6$), $D_l$ ($l\leq7$), $E_6$ or $G_2$. When $\Gg$ is of type $A_l$, the degree formula of Theorem~\ref{thm:intro-degree-formula} allows us to recover the degree computations in \cite{KZJ07} and \cite{AJP} that have been recorded in the OEIS sequence A029729. For other types, the result is new and we record the degrees in those cases obtained via Theorem~\ref{thm:intro-degree-formula} in Appendix~\ref{sec:degree-table}.

Our arguments towards the proof of Theorem~\ref{thm:statement} naturally lend themselves to a discussion of the positive characteristic case in Section~\ref{sec:p>0}. Theorem~\ref{thm:statement} almost immediately implies that $\cc_K$ is Cohen-Macaulay over fields $K$ of positive characteristic $p\gg0$, but we believe that an extension of our methods should be able to work for all primes $p\geq 5$. While we do not give a general argument that resolves the positive characteristic case completely, we describe an explicit computational criterion that allows us to check Cohen-Macaulayness at any prime (see Lemma~\ref{lem:p-crit}). We have used our criterion to verify Cohen-Macaulayness and reducedness when $5\leq p\leq 1000$, which (very slightly) improves upon the results of \cite{SY}, where the $p=5$ case is excluded. Furthermore, we provide independent arguments to resolve the `bad' characteristic cases:

\begin{theorem} \label{thm:char2,3}
Suppose $K$ is a field of characteristic $p$. If $p=3$, the commuting scheme $\cc_K(\Gg_2)$ is Cohen-Macaulay and reduced, has degree $1825$ and has the same Hilbert series as the one given in Theorem~\ref{thm:statement}. If $p=2$, the commuting scheme $\cc_K(\Gg_2)$ is reduced but not Cohen-Macaulay, has degree $438$ and Hilbert series $\widetilde{H}(t)(1-t)^{-17}$ where:
\[\widetilde{H}(t):=1+11t+52t^2+134t^3+192t^4+126t^5-10t^6-52t^7-15t^8-t^9.\]
\end{theorem}

The proof of the above theorem when $p=3$ is close in spirit to our computational criterion above, but requires a couple of additional computational inputs. The proof when $p=2$ is based on the accidental isomorphism $\Gg_2\simeq \psl_4$ that holds in characteristic $2$, which makes the analysis of the commuting scheme much more amenable to geometric arguments.

As mentioned already, in order to prove Theorem~\ref{thm:statement}, we found an explicit linear system of parameters for the commuting scheme using Gr\"obner basis computations. These computations were performed using Macaulay2 \cite{M2} and similar computations were also used in the proof of Theorem~\ref{thm:char2,3}. All of our Macaulay2 and Python scripts can be found in the repository \cite{Git} as well as in the arXiv ancillary files.

\subsection{Future directions}
It would be interesting to explore extensions of Theorem~\ref{thm:intro-degree-formula} to cases when the commuting scheme of $\Bb$ is not necessarily a complete intersection, but that of some other parabolic subalgebra $\mathfrak{p}$ is. (Examples of this hypothesis being satisfied are discussed in \cite{GG18}.) Also, in order to uniformly resolve reducedness and Cohen-Macaulayness in all prime characteristics $p\geq 5$, it would be desirable to perform the Gr\"obner basis computations of Proposition~\ref{prop:m2-first}, or a variant thereof, over $\ZZ$ instead of $\FF_p$. In characteristic $2$, the failure of the commuting scheme to be Cohen-Macaulay can be attributed to the presence of irreducible components having unequal dimensions. A more precise study of these components and how they interact might reveal more structure better explaining this failure.

As noted earlier, although the coefficients of the polynomial $H(t)$ stated in Theorem~\ref{thm:statement} are not symmetric, they are almost symmetric. So, even though $\cc_K$ is not Gorenstein, it should be possible to draw slightly weaker conclusions about the geometry of the commuting scheme which might be of interest. Finally, a natural variant of the problems discussed in this paper is to study their respective counterparts for the commuting scheme $\cc_K(G_2)$ of the algebraic group $G_2$, instead of the Lie algebra $\Gg_2$, via computational methods.

\subsection{Structure of the paper} We discuss some background about the Lie algebra $\Gg_2$ and the commuting scheme in Section~\ref{sec:back}, with explicit equations and structure constants described in Appendix~\ref{sec:app1}.

In Section~\ref{sec:degree-f-la}, we prove Theorem~\ref{thm:intro-degree-formula} (see Theorem~\ref{thm:degree-formula}) and use it to compute the degree of the commuting scheme of $\Gg_2$.

In Section~\ref{sec:mainmath}, we prove Theorem~\ref{thm:statement} (see Corollary~\ref{cor:maincor}) by identifying an explicit linear system of parameters for the commuting scheme and analyzing it using the degree computation from the previous section. In Section~\ref{sec:lower}, we also give an independent computational proof of Theorem~\ref{thm:statement} that does not depend upon Theorem~\ref{thm:intro-degree-formula}.

Finally, in Section~\ref{sec:p>0}, we work with fields of positive characteristic $p$. In particular, we provide a computational criterion (see Lemma~\ref{lem:p-crit}) and use it to prove Cohen-Macaulayness and reducedness when $5\leq p\leq 1000$ (see Proposition~\ref{prop:5-1000}). We also resolve the characteristic $3$ case (see Corollaries~\ref{cor:CM-char3} and~\ref{cor:red-char3}) and the characteristic $2$ case (see Propositions~\ref{prop:char2} and~\ref{prop:red-char2}), thus proving Theorem~\ref{thm:char2,3}.

In Appendix~\ref{sec:degree-table}, we give a table of degrees for the commuting scheme obtained using Theorem~\ref{thm:intro-degree-formula}. A list of the Macaulay2 scripts used in the paper as well as their run times has been given in Appendix~\ref{sec:app-m2}.

\subsection{Acknowledgements} This work was partly supported by the Villum Ascending Investigator grant no. $84071$ from Villum Fonden. The author would like to thank Joshua Mundinger for suggesting the study of the scheme $\cc_K$ by computational methods.

\subsection{AI use disclosure}
The author made use of ChatGPT and Claude for help writing and debugging the Macaulay2 scripts used in the paper as well as the Python code for computing degrees using Theorem~\ref{thm:intro-degree-formula}. Both AI agents were also used for proofreading the document, literature review and for brainstorming ideas to resolve the bad characteristic cases. All proofs and text are human-generated and all computational verifications were performed and checked by the author.

\section{Background and notation} \label{sec:back}

We recall some facts about the root system of type $G_2$ and refer the reader to \cite{Hum} for more details. The Cartan matrix of type $G_2$ is:
\[
 C=(c_{i,j})=\begin{pmatrix}2&-3\\-1&2\end{pmatrix}
\]
and the Lie algebra $\Gg_2$ has a Chevalley-Serre presentation which we describe. The Lie algebra $\Gg_2$ over a field $K$ of characteristic zero is generated by elements $e_i, f_i, h_i$ for $i\in \{1,2\}$ subject to the following relations for all $i,j\in \{1,2\}$:
\[[h_i,h_j]=0,\]
\[[h_i,e_j]=c_{i,j}e_j,\qquad [h_i,f_j]=-c_{i,j}f_j,\]
\[[e_i,f_j]=\delta_{i,j}h_{i},\]
\[\mathrm{ad}(e_i)^{1-c_{i,3-i}}e_{3-i}=0,\]
\[\mathrm{ad}(f_i)^{1-c_{i,3-i}}f_{3-i}=0.\]

Then, we have the root space decomposition:
\[\Gg_2=\Hh\oplus(\bigoplus_{\gamma\in \Phi}\Gg_{\gamma}),\]
where $\Hh:=\mathrm{span}\{h_1,h_2\}$ is a Cartan subalgebra, the set $\Phi\sub \Hh^*$ is the set of roots and $\Gg_{\gamma}$ is the root space corresponding to the root $\gamma$. We make a choice of simple roots $\{\alpha,\beta\}$ so that $e_{1}=e_{\alpha}$ and $e_2=e_{\beta}$, etc. Here, $\alpha$ is the short root and $\beta$ is the long root, so that:
\[(\alpha,\alpha)=2,\qquad (\beta,\beta)=6,\]
where $(\cdot,\cdot)$ denotes the invariant inner product on the root space, normalized so that the short roots have squared length $2$.

The set of roots $\Phi$ has size $12$, which implies that as a $K$-vector space, the Lie algebra $\Gg_2$ is $14$-dimensional. We define the positive root vectors:
\begin{align*}
 e_{\alpha+\beta}&=[e_\alpha,e_\beta],&
 e_{2\alpha+\beta}&=\tfrac12[e_\alpha,e_{\alpha+\beta}],\\
 e_{3\alpha+\beta}&=\tfrac13[e_\alpha,e_{2\alpha+\beta}],&
 e_{3\alpha+2\beta}&=[e_\beta,e_{3\alpha+\beta}].
\end{align*}
Similarly, we have the negative root vectors:
\begin{align*}
 f_{\alpha+\beta}&=[f_\alpha,f_\beta],&
 f_{2\alpha+\beta}&=\tfrac12[f_\alpha,f_{\alpha+\beta}],\\
 f_{3\alpha+\beta}&=\tfrac13[f_\alpha,f_{2\alpha+\beta}],&
 f_{3\alpha+2\beta}&=[f_\beta,f_{3\alpha+\beta}].
\end{align*}
Then, we can fix an ordered basis of $\Gg_2$ given by:
\[(b_i)_{1\leq i\leq 14}:=(h_{\alpha}, h_{\beta}, e_\alpha,e_\beta,e_{\alpha+\beta},e_{2\alpha+\beta},e_{3\alpha+\beta},e_{3\alpha+2\beta},f_\alpha,f_\beta,f_{\alpha+\beta},f_{2\alpha+\beta},
 f_{3\alpha+\beta},f_{3\alpha+2\beta}).\]

The resulting structure constants $c_{ij}^k$ are integral:
\[[b_i, b_j]=\sum_k c_{ij}^k b_k\]
and have been noted explicitly in Appendix~\ref{sec:app1}. For fields $K$ of characteristic $p>0$, we define $\Gg_2$ over $K$ to be precisely the Lie algebra spanned by a basis $(b_i)_{1\leq i\leq 14}$ and the above structure constants.

\begin{definition}
The commuting scheme $\cc_K$ of $\Gg_2$ defined over a field $K$ is the closed subscheme of $\Gg_2\times \Gg_2$ defined by the vanishing of the Lie bracket on $\Gg_2$. More precisely, we define $\cc_K:=\Spec(\aaa_K)$, where
\[\aaa_K:=\frac{K[x_1,x_2,\dots, x_{14},y_1, y_2, \dots, y_{14}]}{\Big(\sum\limits_{i,j=1}^{14}c_{ij}^kx_iy_j: 1\leq k\leq 14\Big)}.\]
\end{definition}
For ease of notation, we will write $\cc:=\cc_{\QQ}$ and $\aaa:=\aaa_{\QQ}$. We will also abuse notation and use the same formula above to define $\aaa_R$ when $R$ is an arbitrary commutative ring.

Now suppose $K$ is a field of characteristic $\neq 2,3$. It is known that $\cc_K$ is irreducible and has dimension $\dim\Gg_2+\dim\Hh=16$ by \cite{Ri} in characteristic zero and by \cite{Levy02} in good positive characteristic. (Richardson actually worked over algebraically closed $K$, but the claim for a general $K$ follows from the observation that the map $\cc_{\ol{K}}\to \cc_K$ is surjective.)

Note that $\cc_K$ is generically reduced over an algebraically closed field $K$ (under our characteristic assumption above) because of the following geometric argument: The commuting scheme $\cc_K$ is the zero fiber of the moment map
\[\mu:\Gg_2\times \Gg_2\to \Gg_2\]
\[(x,y)\mapsto [x,y].\]
The derivative of this map $d\mu$ at any point $(x,y)\in \cc_K$ and any vector $(a,b)$ is given by the vector:
\[[x,b]+[a,y].\]
When $x$ is regular semisimple, we can take $a=0$ to conclude that the derivative $d\mu$ has rank at least $\dim\Gg_2-\dim\Hh=12$ at the point $(x,y)$ showing that the tangent space to $\cc_K$ at the point has dimension at most $28-12=16$. However, as $\dim(\cc_K)=16$, this shows that $(x,y)$ is a smooth point of $\cc_K$. Therefore, since $\cc_K$ is irreducible, we conclude that it is generically smooth, and thus, generically reduced. (The same argument actually works for any reductive Lie algebra in characteristic zero.)

We end the section with some elementary observations that will be useful for computations later. Recall the non-degenerate Killing form $\kappa(x,y):=\Tr(\ad x \ad y)$ for $x,y\in \Gg_2$. Let $Q_k(x,y):=\sum\limits_{i,j=1}^{14} c_{ij}^kx_iy_j \in \ZZ[x,y]$ be the quadratic polynomials that define the scheme $\cc$.

\begin{lemma} \label{lem:killing}
We have the following identities in $\ZZ[x,y]$:
\[\sum_{k,l=1}^{14}\kappa(b_k,b_l)x_k Q_l(x,y)=0,\]
\[\sum_{k,l=1}^{14}\kappa(b_k,b_l)y_k Q_l(x,y)=0.\]
\end{lemma}

\begin{proof}
The above identities are respectively equivalent to the equalities $\kappa(x,[x,y])=0$ and $\kappa(y,[x,y])=0$, which are immediate from the invariance of the Killing form.
\end{proof}

Let $\lambda_l:=\sum\limits_{k=1}^{14} \kappa(b_k,b_l)x_k = \kappa(x,b_l)$ and $\mu_l:=\sum\limits_{k=1}^{14} \kappa(b_k,b_l)y_k = \kappa(y,b_l)$ be linear forms. Then, we can eliminate $Q_1$ and $Q_2$ from the above two equations to get:
\begin{corollary} \label{cor:genci}
We have the following identities in $\ZZ[x,y]$:
\[-\Delta Q_1(x,y)=\sum_{l=3}^{14}(\lambda_l\mu_2-\lambda_2\mu_l)Q_l(x,y),\]
\[\Delta Q_2(x,y)=\sum_{l=3}^{14}(\lambda_l\mu_1-\lambda_1\mu_l)Q_l(x,y)\]
where $\Delta=\lambda_1\mu_2-\lambda_2\mu_1$.
\end{corollary}

(See Appendix~\ref{sec:app1} for an explicit computation of $\Delta$.)

\begin{remark} \label{rem:genci}
The above corollary has the following geometric interpretation: In the coordinate ring of the basic open subscheme $(\Gg_2\times \Gg_2)_{\Delta}$ of $\Gg_2\times \Gg_2$ characterized by $\Delta\neq 0$, it follows that the polynomials $Q_1$ and $Q_2$ lie in the ideal generated by $Q_i$'s for $3\leq i\leq 14$. Thus, the commuting ideal is generated by these $12$ polynomials, and since $28-12=16=\dim(\cc)=\dim(\cc_{\Delta})$, we conclude that $\cc_{\Delta}$ is a complete intersection and that $\cc$ is generically a complete intersection. (Note that $\cc_{\Delta}$ is non-empty as it contains the pair $(h_{\alpha},h_{\beta})$ where $\Delta=192$.) Thus, the two Killing-form relations above account for the discrepancy between the $14$ defining equations and the codimension $12$ of the commuting scheme.
\end{remark}

\section{A root theoretic formula for the degree} \label{sec:degree-f-la}

We start by setting up some notation. Let $K$ be an algebraically closed field of characteristic $0$. Let $G$ be a connected reductive group over $K$ with Lie algebra $\Gg$ and let $H\subseteq B\subseteq G$ be a
maximal torus and a Borel subgroup, with Lie algebras $\Hh\sub\Bb
=\mathfrak h\oplus\Nn$. Let $\Phi^+$ be the set of positive roots, $N:=|\Phi^+|$ be the number of positive roots
and $r:=\mathrm{rank}(\Gg)$ be the rank of $\Gg$. Let $W$ be the Weyl group.

We consider the group $\widetilde H:=H\times\GG_m$ acting on $V:=\Gg\oplus\Gg$, with $H$ acting
adjointly and $\GG_m$ acting diagonally by scalar multiplication. Let $X(\widetilde H)\simeq X(H)\oplus \ZZ A\simeq \ZZ^{r+1}$ denote the space of weights, where $A$ is the weight of the scalar action. Define $S:=\mathrm{Sym}_{\QQ}(X(H))$ to be the polynomial ring in $r$ variables with a natural $W$-action. We will often view it as a subring of $S[A]=\mathrm{Sym}_{\QQ}(X(\widetilde{H}))$ which is a polynomial ring in $r+1$ generators.

The coordinate ring $K[V]=\mathrm{Sym}(V^*)$ is graded by $X(\widetilde H)$: a coordinate function dual to a vector of weight $\mu$ has degree $-\mu$. (For $\Gg=\Gg_2$, this is the
grading of the algebra $K[x_i,y_j]$ in which $x_i$ and $y_i$ have degree $-(\mathrm{wt}(b_i)+A)$ for the basis
$(b_i)$ of Section~\ref{sec:back}.) 

\subsection{Hilbert series and characters} \label{sec:Hilb-char} For a finitely generated
$X(\widetilde H)$-graded $K[V]$-module $M$ we define its character:
\[\mathrm{ch}(M):=\sum_\lambda\dim M_\lambda\,e^\lambda\in \ZZ[[X(\widetilde{H})]].\]
For later reference, we note that there exists a ring homomorphism $\ZZ[X(\widetilde H)] \to S[A][[s]]$ given by the assignment $e^{\lambda}\mapsto \mathrm{exp}(s\lambda)=\sum_{k\ge0}s^k\lambda^k/k!\in S[A][[s]]$. Furthermore, this homomorphism extends to one from the localization of $\mathbb Z[X(\widetilde H)]$ at the elements
$1-e^{\lambda}$ (for $\lambda\neq0$) to $\mathrm{Frac}(S[A])[[s]][s^{-1}]$.

Taking $M=K[V]$, we note that $\mathrm{ch}(K[V])=\prod_{\mu\in\operatorname{wt}(V)}\bigl(1-e^{-\mu}\bigr)^{-1}$. By Hilbert's syzygy theorem, for any such $K[V]$-module $M$, we have that $\ch(M)=\kk_M\cdot\ch(K[V])$, where  $\kk_M\in\mathbb Z[X(\widetilde H)]$ is a Laurent polynomial.

Mapping $X(H)\to0$ and $e^{-A}\mapsto t$
turns $\ch(M)$ into the ordinary Hilbert series $\kk_M(t)/(1-t)^{\dim V}$ of $M$ for the standard
grading on $K[V]$. Suppose the Krull dimension of the $K[V]$-module $M$ is equal to $m$. Then, there exists a Laurent polynomial $\widetilde\kk_M(t)$ such that:
\begin{itemize}
\item The Hilbert series $\mathrm{HS}(M)$ is equal to $\widetilde\kk_M(t)/(1-t)^m$. (So, we have the equality $\widetilde\kk_M(t)=\kk_M(t)(1-t)^{m-\dim(V)}$.)
\item The degree of $M$ (denoted by $e(M)$) is equal to $\widetilde\kk_M(1)$.
\end{itemize}
(The polynomial $H(t)$ in Theorem~\ref{thm:statement} is precisely the polynomial $\widetilde\kk_M(t)$ when $M=\aaa$.) We will also be computing the degree of $M$ in a different way, using the following associativity formula:
\begin{equation}\label{eq:associativity}
e(M)=\sum_{\mathfrak p}\ \ell_{K[V]_{\mathfrak p}}(M_{\mathfrak p})\cdot e\bigl(K[V]/\mathfrak p\bigr),
\end{equation}
where the sum is over the primes $\mathfrak p$ such that $\dim(K[V]/\mathfrak p)=m$ (see \cite[Corollary 4.6.8]{BH93}). Here, the quantity $\ell_{K[V]_{\mathfrak p}}(M_{\mathfrak p})$ denotes the length of $M_{\mathfrak p}$ as a $K[V]_{\mathfrak p}$-module.

Finally, we recall the Borel-Weil-Bott theorem and a character identity that follows from it (see \cite[Part II-5]{Jant} for details). For any integral weight $-\lambda$, there exists a one-dimensional $B$-module $K_{-\lambda}$, which is obtained by pulling back the $H$-action along the surjection $B\twoheadrightarrow H$. Let $\mathcal L_{\lambda}:=G\times_B K_{-\lambda}$ be the associated
$G$-equivariant vector bundle on $G/B$. Then, the Borel-Weil-Bott theorem says that if $\lambda+\rho$ is a regular weight, then:
\[H^i(G/B, \mathcal{L}_\lambda)\simeq\begin{cases} (L_{w\cdot \lambda})^{\vee} & \text{if } i=\ell(w)\\
0& \text{otherwise}.
\end{cases}\]
Here, $w\in W$ is the unique Weyl group element such that $w\cdot \lambda:=w(\lambda+\rho)-\rho$ is dominant (where $\rho$ is half the sum of the positive roots), $\ell(\cdot)$ is the length function on the Weyl group and $L_{w\cdot \lambda}$ is the unique finite dimensional $G$-representation with highest weight $w\cdot \lambda$.  By the Weyl character formula, we have that:
\begin{align*}
\ch((L_{w\cdot \lambda})^{\vee})&=\frac{\sum\limits_{w'\in W}(-1)^{\ell(w')}e^{-w'(w(\lambda+\rho))}}{\sum\limits_{w''\in W}(-1)^{\ell(w'')}e^{-w''(\rho)}}\\
&=\frac{(-1)^{\ell(w)}\sum\limits_{w'\in W}(-1)^{\ell(w')}e^{-w'(\lambda+\rho)}}{\sum\limits_{w''\in W}(-1)^{\ell(w'')}e^{-w''(\rho)}}\\
&=(-1)^{\ell(w)}\sum\limits_{w'\in W}w'\Big(\frac{e^{-(\lambda+\rho)}}{\sum\limits_{w''\in W}(-1)^{\ell(w'')}e^{-w''(\rho)}}\Big)\\
&=(-1)^{\ell(w)}\sum\limits_{w'\in W}w'\Big(\frac{e^{-\lambda}}{\prod\limits_{\gamma\in \Phi^{+}}(1-e^{\gamma})}\Big).
\end{align*}

On the other hand, if $\lambda+\rho$ is a non-regular weight, no $w$ exists such that $w\cdot \lambda$ is dominant. This happens when $\lambda+\rho$ lies on one of the root hyperplanes. In this case, the Borel-Weil-Bott theorem says that all the cohomology groups $H^i(G/B,\mathcal{L}_{\lambda})$ vanish and so does the sum in the above Weyl character identity.

More generally, given any finite dimensional $B$-module $M$, we can write a composition series for it in terms of the above $K_{\lambda}$'s. In that case, if we consider the associated bundle $\mathcal{L}(M):=G\times_B M$ and the Euler characteristic $\chi(M):=\sum_i(-1)^i\ch(H^i(G/B,\mathcal L(M)))$, then the above character formula implies by linearity that:
\begin{equation}\label{eq:weyl}
\chi(M)=\sum_{w\in W}w\!\left(\frac{\ch(M)}{\prod_{\gamma\in\Phi^+}(1-e^{\gamma})}\right).
\end{equation}
If $M$ is moreover an $\widetilde H$-module on which $\mathbb G_m$ acts
through a character, \eqref{eq:weyl} holds verbatim for $\widetilde H$-characters.

\subsection{A root theoretic formula for the degree of $\cc(\Gg)$} \label{sec:degree-comp}
Now we state the main theorem of this section. Since $[\Bb,\Bb]=\Nn$, the commuting scheme $\cc(\Bb)$ is cut out in $\Bb\times \Bb$ precisely by the $\dim(\Nn)=\dim(\Bb)-r$ equations. In particular, $\cc(\Bb)$ is a complete intersection precisely when its dimension is $\dim(\Bb)+r$.

\begin{theorem}\label{thm:degree-formula}
Suppose the commuting scheme $\cc(\Bb)$ of the Borel subalgebra has dimension $\dim(\Bb)+r$. Let
\[
P:=\prod_{\gamma\in\Phi^+}(\gamma-1)^2(\gamma+2)\in S
\]
and let $P_N$ be its homogeneous component of degree $N$. Then
\[
\deg(\cc(\Gg))=\frac{(-1)^N}{|W|}\frac{1}{\prod_{\gamma\in\Phi^+}\gamma}\sum_{w\in W} (-1)^{\ell(w)}w(P_N).
\]
\end{theorem}

\begin{remark} \label{rem:degree-hypo}
By \cite[Theorem 1.3]{GG18}, the commuting scheme $\cc(\Bb)$ is irreducible whenever every simple factor of $\Gg$ is of type $A_l$ ($l\leq 15$), $B_l$ ($l\leq6$), $C_l$ ($l\leq6$), $D_l$ ($l\leq7$), $E_6$ or $G_2$. In that case, by \cite[Theorem 1.1]{GG18}, we have the equality $\dim(\cc(\Bb))=\dim(\Bb)+\dim(\Hh)$, and therefore the hypothesis of the above theorem is satisfied whenever every simple factor of $\Gg$ is of the above form. 

In fact, it is possible to show that the complete intersection hypothesis of the above theorem holds also when any of the simple factors are $B_7, C_7$ or $F_4$, even though $\cc(\Bb)$ is not irreducible in these cases. This follows from the tables in \cite[$\mathsection 7$]{GG18}, but we do not pursue that here.
\end{remark}

We provide a proof of the above theorem in Section~\ref{sec:degree-proof}. Right now, we suppose $\Gg=\Gg_2$ and use the above formula to compute the degree of $\cc(\Gg_2)$. The same computation has been replicated using a Python script for $\Gg_2$ and other root systems satisfying the hypothesis of the above theorem, and the results have been noted in Appendix~\ref{sec:degree-table}.

When $\Gg=\Gg_2$, we have $N=6$, $|W|=12$, $\Phi^+=\{\alpha,\beta,\alpha+\beta,2\alpha+\beta,3\alpha+\beta,3\alpha+2\beta\}$ and $S=\mathbb Q[\alpha,\beta]$. Expanding $P=\prod_{\gamma\in \Phi^+}(\gamma-1)^2(\gamma+2)$, we compute (using WolframAlpha) that the degree $6$ component is equal to:
\[P_6=-107728\alpha^6-434478\alpha^5\beta-675699\alpha^4\beta^2-548376\alpha^3\beta^3-250305\alpha^2\beta^4-60462\alpha\beta^5-5872\beta^6.\]

Next, we need to compute:
\[ \frac{1}{\prod_{\gamma\in\Phi^+}\gamma}\sum_{w\in W} (-1)^{\ell(w)} w(P_N).\]
By \cite[Proposition 3]{Dema}, the above expression is exactly $D_{w_0}(P_N)$, where $D_{w_0}$ is the Demazure operator associated with the longest element $w_0$ in the Weyl group $W$. Then, if we denote by $s_{\alpha}, s_{\beta}\in W$ the simple reflections with respect to $\alpha$ and $\beta$ respectively, we have the equality $w_0=s_{\alpha}s_{\beta}s_{\alpha}s_{\beta}s_{\alpha}s_{\beta}$. In that case, we have that $D_{w_0} = D_{\alpha}D_{\beta}D_{\alpha}D_{\beta}D_{\alpha}D_{\beta}$, where for any root $\gamma$, we define the following divided difference operator on $S$:
\[D_{\gamma}(f)=\frac{f-s_{\gamma}(f)}{\gamma}.\]
Note that we have the equalities:
\[s_{\alpha}(\alpha)=-\alpha, \quad s_{\alpha}(\beta)=3\alpha+\beta,  \quad s_{\beta}(\alpha)=\alpha+\beta,\quad s_{\beta}(\beta)=-\beta.\]
Using these formulae, we can successively apply the divided difference operators:

\begin{align*}
D_{\beta}(P_6)&=-222588\alpha^5-556470\alpha^4\beta-584176\alpha^3\beta^2-319794\alpha^2\beta^3-88668\alpha\beta^4-9584\beta^5,\\
D_{\alpha}D_{\beta}(P_6)&=-252108\alpha^4-565242\alpha^3\beta-490598\alpha^2\beta^2-201456\alpha\beta^3-33576\beta^4,\\
D_{\beta}D_{\alpha}D_{\beta}(P_6)&=-122052\alpha^3-183078\alpha^2\beta-109010\alpha\beta^2-23992\beta^3,\\
D_{\alpha}D_{\beta}D_{\alpha}D_{\beta}(P_6)&=-28176\alpha^2-6276\alpha\beta-2092\beta^2,\\
D_{\beta}D_{\alpha}D_{\beta}D_{\alpha}D_{\beta}(P_6)&=43800\alpha+21900\beta,\\
D_{\alpha}D_{\beta}D_{\alpha}D_{\beta}D_{\alpha}D_{\beta}(P_6)&=21900.
\end{align*}
Hence $\deg(\cc(\Gg_2))=D_{w_0}(P_6)/|W|=21900/12=1825$.

\subsection{Proof of Theorem~\ref{thm:degree-formula}} \label{sec:degree-proof}

The strategy of the proof will be as follows: We will first compute the character of the commuting scheme $\cc(\Bb)$. Next, we will relate this character to the Euler characteristic of a certain Grothendieck-Springer type construction. Finally, we will show that the leading term of this Euler characteristic precisely computes the degree of $\cc(\Gg)$.

Since $[\Bb,\Bb]\sub\Nn$, the scheme
$\cc(\Bb)\subseteq\Bb\oplus\Bb$ is cut out by $N$ equations given by the vanishing of the $\gamma$-coordinate of $[x,y]$ for each $\gamma\in \Phi^+$. These equations then have degree $-(\gamma+2A)$. By assumption, we have that $\dim(\cc(\Bb))=\dim(\Bb)+r=2\dim(\Bb)-N$, and so, these equations form
a regular sequence in the polynomial ring $K[\Bb\oplus\Bb]$. Therefore, we have that:
\begin{equation}\label{eq:koszul}
\ch(K[\cc(\mathfrak b)])=\prod_{\gamma\in\Phi^+}(1-e^{-(\gamma+2A)})\cdot\ch(K[\mathfrak b\oplus\mathfrak b]).
\end{equation}

Next, define $\xx:=G\times_B \cc(\Bb)$ which is a bundle over $G/B$ via the projection
$p:\xx\to G/B$ on the first coordinate, such that the fiber is $\cc(\Bb)$. Furthermore, we have another map:
\[\pi:\xx\to V,\]
\[(g,(x,y))\mapsto(\mathrm{Ad}_gx,\mathrm{Ad}_gy),\]
where recall that $V=\Gg\oplus \Gg$. Both $p$ and $\pi$ are $\widetilde H$-equivariant (where $H$ acts by left multiplication on $G/B$). Also, the map $\pi$ is proper since it factors through $G/B\times V$.

Let $K[\cc(\mathfrak b)]=\bigoplus_{d\ge0}M_d$, where $M_d$ is the space of functions of degree $d$. Then, the pushforward sheaf $p_*\oo_{\xx}$ is equal to the direct sum of associated bundles $\bigoplus_d\mathcal L(M_d)$. Therefore, we have the equality:
\[H^i(\xx,\oo_{\xx})=\bigoplus_dH^i(G/B,\mathcal L(M_d)).\]
Then, applying Equation~\eqref{eq:weyl} to each $M_d$ gives
\begin{equation}\label{eq:euler}
\sum_{i\geq 0}(-1)^i\ch(H^i(\xx,\oo_{\xx}))
=\sum_{w\in W}\sum_{d}w\left(\frac{\ch (M_d)}{\prod_{\gamma\in\Phi^+}(1-e^{\gamma})}\right)=\sum_{w\in W}w\left(\frac{\ch (K[\cc(\Bb)])}{\prod_{\gamma\in\Phi^+}(1-e^{\gamma})}\right).
\end{equation}

Next, we note that the image of $\pi:\xx\to V$ is exactly $\cc(\Gg)$, since every pair of commuting elements in $\Gg$ can be simultaneously conjugated into a Borel subalgebra. Let $U\sub\cc(\Gg)$ be the open set of pairs $(x,y)$ with $x$ regular semisimple, which is dense since $\cc(\Gg)$ is irreducible by \cite{Ri}. By \cite[Proposition 3.1.29]{CG}, any point in $U$ has exactly $|W|$ pre-images under $\pi$. Using the tangent space argument of Section~\ref{sec:back} for general $\Gg$ instead of $\Gg_2$, we conclude that $U$ is smooth of dimension $\dim(\Gg)+r$. Again, by a similar argument, we get that the derivative $d\pi$ is surjective (and hence, an isomorphism), thus establishing that $\pi^{-1}(U)$ is also smooth and the map $\pi:\pi^{-1}(U)\to U$ is finite \'etale of degree $|W|$.

Therefore, we can conclude that the higher direct images $R^i\pi_*\oo_{\mathcal X}|_U=0$ for $i>0$. Since $R^i\pi_*\oo_{\mathcal X}$ is a coherent sheaf over $\cc(\Gg)$ associated with the $K[\cc(\Gg)]$-module $\Gamma(\cc(\Gg),R^i\pi_*\oo_{\mathcal X}) \simeq H^i(\xx, \oo_\xx)$, this implies that $\dim(H^i(\xx,\oo_{\mathcal X}))<\dim(\cc(\Gg))$ for all $i>0$. Furthermore, the pushforward $(\pi_*\oo_{\xx})|_U$ is a locally free sheaf of rank $|W|$.

Let $\mathfrak{p}\subset K[V]$ be the prime ideal defining
the underlying reduced scheme $\cc(\Gg)_{\mathrm{red}}$. The localization $H^0(\xx,\oo_{\xx})_{\mathfrak p}$ is the
stalk of $\pi_*\oo_{\xx}$ at the generic point of $\cc(\Gg)$, which lies in $U$. Hence, it is a
free $K[\cc(\Gg)]_{\mathfrak p}$-module of rank $|W|$. In particular, we have the equality $\dim H^0(\xx,\oo_{\xx})=\dim(\cc(\Gg))$. Applying Equation~\eqref{eq:associativity} to the $K[V]$-modules $H^0(\xx,\oo_{\xx})$
and $K[\cc(\Gg)]$, and using $\ell_{K[V]_{\mathfrak p}}(H^0(\xx,\oo_{\xx})_{\mathfrak p})
=|W|\cdot\ell_{K[V]_{\mathfrak p}}(K[\cc(\Gg)]_{\mathfrak p})$, we get that:
\[e(H^0(\xx,\oo_{\xx}))=|W|\cdot e(K[\cc(\Gg)]).\]
We can now frame the above observations in terms of the Hilbert numerators. Firstly, since $\dim(H^0(\xx,\oo_{\xx}))=\dim(\cc(\Gg))$, we have that:
\[e(H^0(\xx,\oo_{\xx}))=\lim_{t\to 1}(\kk_{H^0(\xx,\oo_{\xx})}(t)(1-t)^{\dim(\cc(\Gg))-\dim(V)})=\lim_{t\to 1}(\kk_{H^0(\xx,\oo_{\xx})}(t)(1-t)^{-2N}).\]
Similarly, since $\dim(H^i(\xx,\oo_{\xx}))<\dim(\cc(\Gg))$ for all $i>0$, we have that:
\[0=\lim_{t\to 1}(\kk_{H^i(\xx,\oo_{\xx})}(t)(1-t)^{\dim(\cc(\Gg))-\dim(V)})=\lim_{t\to 1}(\kk_{H^i(\xx,\oo_{\xx})}(t)(1-t)^{-2N}).\]
Therefore, substituting $t=e^{-A}$ and expanding via the ring homomorphism mentioned at the beginning of Section~\ref{sec:Hilb-char}, we get that:
\begin{align}
\sum_{i\geq 0}(-1)^i\kk_{H^i(\xx,\oo_{\xx})}(\exp(-sA))&= (1-\exp(-sA))^{2N}e(H^0(\xx,\oo_{\xx}))+O((1-\exp(-sA))^{2N+1})\notag\\
&=s^{2N}A^{2N}|W|e(K[\cc(\Gg)]) + O(s^{2N+1}). \label{eq:leftside}
\end{align}

Finally, we compare this expression with one we obtain directly from Equation~\eqref{eq:euler}. We divide \eqref{eq:euler} by $\ch(K[V])$, use the equality from Equation~\eqref{eq:koszul} and the fact that the weights in $V$ not already included in $\Bb\oplus \Bb$ are exactly $-\gamma+A$ where $\gamma\in\Phi^+$,
each taken twice, to obtain the identity
\begin{equation}\label{eq:identity}
\sum_{i\geq0}(-1)^i\kk_{H^i(\xx,\oo_{\xx})}
=\sum_{w\in W}w\left(\frac{\prod_{\gamma\in\Phi^+}\bigl(1-e^{\gamma-A}\bigr)^2\bigl(1-e^{-\gamma-2A}\bigr)}{\prod_{\gamma\in\Phi^+}(1-e^{\gamma})}\right).
\end{equation}
Now, expanding via the ring homomorphism of Section~\ref{sec:Hilb-char}, the left hand side of the above identity becomes a power series $\sum_{m\geq 0}c_ms^m$ with $c_m\in S[A]$. In the $w$-th summand on the right, the numerator is $s^{3N}w\bigl(\prod_{\gamma\in \Phi^+}(\gamma-A)^2(\gamma+2A)\bigr)(1+O(s))$ and the
denominator is $s^{N}w\bigl(\prod_{\gamma\in \Phi^+}(-\gamma)\bigr)(1+O(s))$. Thus, comparing coefficients, we get that $c_m=0$ when $m<2N$ and
\[c_{2N}=\sum_{w\in W}w\left(\frac{\prod_{\gamma\in\Phi^+}(\gamma-A)^2(\gamma+2A)}{\prod_{\gamma\in\Phi^+}(-\gamma)}\right).
\]
Note that Equation~\eqref{eq:leftside} computes precisely the limit of Equation~\eqref{eq:identity} as $X(H)\to 0$. Finally, we can also let $A=1$ to conclude that:
\[\deg(\cc(\Gg))=e(K[\cc(\Gg)])=\frac{1}{|W|}\sum_{w\in W}w\left(\frac{\prod_{\gamma\in\Phi^+}(\gamma-1)^2(\gamma+2)}{\prod_{\gamma\in\Phi^+}(-\gamma)}\right)\Bigg|_{\ X(H)\mapsto0}.\]
In particular, to compute that limit as $X(H)\to 0$, we can replace the numerator with its degree $N$ part, since that is also the degree of the denominator. Finally, we note that for any $w\in W$, we have theq equality $w\left(\prod_{\gamma\in\Phi^+}\gamma\right)=-1^{\ell(w)}\prod_{\gamma\in\Phi^+}\gamma$, thus completing the proof of the theorem.

\section{Proof of Theorem~\ref{thm:statement}}\label{sec:mainmath}

In this section, we suppose that $K$ is an arbitrary field of characteristic zero. For ease with Macaulay2 computations, we will often work in this section over the reduction of the relevant objects to the finite field $\FF_p$ where $p=32003$. Also, for the rest of this section, we write $\Lambda:=\ZZ_{(p)}$ for the localization of $\ZZ$ at the ideal $(p)$.

\subsection{System of parameters and Hilbert series} \label{sec:sop}

We define the following permutation $\sigma$ on the set $\{1,2,\dots, 14\}$:
\[\begin{array}{c|cccccccccccccc}
i&1&2&3&4&5&6&7&8&9&10&11&12&13&14\\ \hline
\sigma(i)&10&14&6&7&2&13&12&11&3&4&1&9&8&5
\end{array}\]

(Some motivation for the above definition is given in Remark~\ref{rem:perm}.) We use this permutation above to define a system of parameters for the algebra $\aaa$. To that end, define the following elements in $\aaa$:
\[s_i:=x_i-y_{\sigma(i)} \qquad (1\leq i\leq 14), \qquad s_{15}=x_1, \qquad s_{16}=x_5.\]
We also fix the elements:
\[(u_1,u_2,\dots,u_{12})=(x_2,x_3,x_4,x_6,x_7,x_8,x_9,x_{10},x_{11},x_{12},x_{13},x_{14}).\]
Note that the elements $\{s_i, u_j: 1\leq i\leq 16, 1\leq j\leq 12\}$ are obtained by a linear change of the coordinates $\{x_i,y_j:1\leq i,j\leq 14\}$ and generate the polynomial ring $\QQ[x_i,y_j: 1\leq i,j\leq 14]$ and so, we can write:
\[\aaa=\QQ[s_i, u_j: 1\leq i\leq 16, 1\leq j\leq 12]/(P_k(s_i,u_j):1\leq k\leq 14),\]
where the polynomials $P_k(s_i,u_j)$ are exactly the defining commuting equations of the algebra $\aaa$ expressed in terms of the new coordinates we have introduced. (That is, the polynomials $P_k$ are obtained by a linear change of coordinates from the polynomials $Q_k$ defined in Section~\ref{sec:back}.)

Now, we consider the $\FF_p$-algebra $\aaa_{\FF_p}$ and the reductions $\ol{s_i}$, $\ol{u_j}$ and $\ol{P_k}(\ol{s_i},\ol{u_j})$ of the above $s_i$, $u_j$ and $P_k(s_i,u_j)$ respectively modulo $p$.

\begin{proposition} \label{prop:m2-first}
The quotient $\aaa_{\FF_p}/(\ol{s_1},\ol{s_2},\dots, \ol{s_{16}})$ is finite dimensional as an $\FF_p$-vector space. Furthermore, the Hilbert series of this quotient $\mathrm{HS}(\aaa_{\FF_p}/(\ol{s_1},\ol{s_2},\dots, \ol{s_{16}}))$ is
\[H(t)= 1 + 12t + 64t^2 + 198t^3 + 388t^4 + 492t^5 + 395t^6 + 198t^7 + 64t^8 + 12t^9 + t^{10}.\]
In particular, $\dim_{\FF_p}(\aaa_{\FF_p}/(\ol{s_1},\ol{s_2},\dots, \ol{s_{16}}))=H(1)=1825.$
\end{proposition}

The above proposition follows by a Gr\"obner basis computation using the Macaulay2 file \texttt{system\_of\_parameters.m2} of \cite{Git}.

\begin{corollary} \label{cor:upper}
The elements $\{s_i\}_{1\leq i\leq 16}$ form a system of parameters for the algebra $\aaa$. Furthermore, as a $T:=\QQ[s_1,s_2,\dots, s_{16}]$-module,  $\aaa$ is generated by at most $1825$ elements.
\end{corollary}

\begin{proof}
Proposition~\ref{prop:m2-first} implies that there is a set of monomials $\ol\bb$ in the algebra $\aaa_{\FF_p}$ such that $|\ol{\bb}|=1825$ and their images span $\aaa_{\FF_p}/(\ol{s_1},\ol{s_2},\dots, \ol{s_{16}})$ as an $\FF_p$-vector space. Since there exists a surjection $\aaa_{\Lambda} \twoheadrightarrow \aaa_{\FF_p}$, we can lift the elements of $\ol\bb$ to monomials in $\aaa_{\Lambda}$, which can then be viewed as a set of monomials $\bb$ in $\aaa$, since $\Lambda$ is a subring of $\QQ$.

Fix a non-negative integer $d$. Then, the degree $d$ component of the algebra $\aaa_{\Lambda}/(s_1,s_2,\dots, s_{16})$ is spanned by monomials in $u_i$'s of degree $d$ and is finitely generated as a $\Lambda$-module. Then, since $\Lambda$ is local with maximal ideal $(p)$, by Nakayama's lemma applied to each of these finitely generated modules, we conclude that the set of monomials $\bb$ generates $\aaa_{\Lambda}/(s_1,s_2,\dots, s_{16})$ as a $\Lambda$-module. The claim for $\QQ$ follows from this immediately.
\end{proof}

\begin{proposition}\label{cor:CM-alt}
$\aaa$ is a free $T$-module of rank $1825$.
\end{proposition}

\begin{proof}
By Corollary~\ref{cor:upper} there is a surjection $\phi\colon T^{1825}\twoheadrightarrow\aaa$ of
$T$-modules. We write a free resolution for $\aaa$ as a graded $T$-module:
\[0\to F_{16}\to\cdots\to F_0\to\aaa\to0.\]
(The resolution has length at most $16$ by Hilbert's syzygy theorem.) Suppose $\mathrm{HS}(\aaa)=Q(t)/(1-t)^{16}$. Tensoring the above resolution by the flat $T$-module $\mathcal{L}=\mathrm{Frac}(T)$, we get the equality:
\[Q(1)=\sum\limits_{i=0}^{16}(-1)^i \rank(F_i) = \dim_{\mathcal{L}}(\aaa\otimes_T \mathcal{L}).\]
On the other hand $Q(1)=\deg(\cc)=1825$ by the computation in Section~\ref{sec:degree-comp}. Thus $\phi\otimes_T\mathcal L$ is an
isomorphism of $\mathcal{L}$-vector spaces. Therefore, since $T^{1825}$ is torsion-free, we conclude that $\phi$ must have been an isomorphism to start with, thus completing the proof.
\end{proof}

Corollary~\ref{cor:upper} and Proposition~\ref{cor:CM-alt} together imply the following:

\begin{corollary}
We have a graded free decomposition of $\aaa$ as a $T$-module:
\begin{align*}
\aaa \simeq T \oplus T[-1]^{12} \oplus T[-2]^{64} \oplus T[-3]^{198} \oplus T[-4]^{388}&\\\oplus T[-5]^{492}\oplus T[-6]^{395}\oplus &T[-7]^{198}\oplus T[-8]^{64}\oplus T[-9]^{12}\oplus T[-10].
\end{align*}
\end{corollary}

Finally, we prove the main result of this section.

\begin{corollary} \label{cor:maincor}
The scheme $\cc_K$ is reduced, Cohen-Macaulay and normal. Furthermore, the sequence $(s_1,s_2,\dots, s_{16})$ is a regular sequence for $\aaa_K$ and the Hilbert series $\mathrm{HS}(\cc_K)$ is given by $H(t)(1-t)^{-16}$, where $H(t)$ is the polynomial in Theorem~\ref{thm:statement}.
\end{corollary}

\begin{proof}
Note that $\aaa_K=\aaa\otimes_{\QQ} K$. Therefore, Proposition~\ref{cor:CM-alt} implies that $\aaa_K$ is a free module over the polynomial ring $K[s_1,s_2,\dots,s_{16}]$. This implies that $(s_1,s_2,\dots, s_{16})$ is a regular sequence in $\aaa_K$, showing that $\cc_K$ is Cohen-Macaulay. This along with generic reducedness proves that $\cc_K$ is reduced when $K$ is algebraically closed. For general $K$, we note that any nilpotent in $\aaa_K$ naturally also embeds into $\aaa_{\ol{K}}$ as a nilpotent, and thus, we conclude that $\cc_K$ is reduced.

Finally, by \cite[Corollary 1.9(i)]{Pop08}, the scheme $\cc_K$ is regular in codimension $1$, and so, we get that $\cc_K$ is normal by Serre's criterion. The Hilbert series is a consequence of the previous corollary.
\end{proof}

\begin{remark}
The arguments in \cite{Pop08} are over an algebraically closed $K$, but can be extended to a general $K$ since the irregular locus of $\cc_K$ is cut out by certain rank conditions defined over $K$ which remain unchanged when extending to $\ol K$.
\end{remark}

\subsection{Independent computational proof} \label{sec:lower}

In the previous subsection, Proposition~\ref{cor:CM-alt} was the only place where the degree formula of Section~\ref{sec:degree-f-la} was used. In this subsection, we provide a computational proof of Proposition~\ref{cor:CM-alt} that is independent of Theorem~\ref{thm:degree-formula}, thus giving us another proof of Corollary~\ref{cor:maincor}.

To that end, we consider the following vector $c\in\ZZ^{16}$ which was obtained by lifting a randomly sampled vector in $\FF_p^{14}$ (recall that $p=32003$) adjoined with two $1$'s at the end:
\[c = (13641, 22825, 27432, 6173, 21845, 18827, 20908, 8404,
5290, 9258, 14370, 10545, 3220, 26424, 1, 1).\]
(We took the last two coordinates $x_1$ and $x_5$ equal to 1 for convenience and is not essential.)  We write $\ol{c}\in\FF_p^{16}$ for the reduction of $c$ modulo $p$ and define the $\FF_p$-algebra:
\[\aaa_{\FF_p}(\ol{c}):=\FF_p[u_j: 1\leq j\leq 12]/(\ol{P_k}(\ol{c},u_j): 1\leq k\leq 14)=\aaa_{\FF_p}/(\ol{s_1}-\ol{c_1},\ol{s_2}-\ol{c_2},\dots,\ol{s_{16}}-\ol{c_{16}}),\]
which is the fiber of $\aaa_{\FF_p}$ over the $\FF_p$-point $\ol{c}$ of $\Spec(\FF_p[\ol{s_1},\ol{s_2},\dots,\ol{s_{16}}])$. By an explicit computation in Appendix~\ref{sec:app1}, the element $\Delta$ of Corollary~\ref{cor:genci} is given in terms of the coordinates $(s_i,u_j)$, by $\Delta=192(s_{15}(s_{16}-s_5)-u_1(u_9-s_{11}))$.

\begin{proposition} \label{prop:m2-second}
The $\FF_p$-algebra $\aaa_{\FF_p}(\ol{c})$ satisfies the following two conditions:
\begin{enumerate}[label=(\alph*)]
\item $\dim_{\FF_p}(\aaa_{\FF_p}(\ol{c}))=1825$.
\item The image of $\Delta$, namely $192(\ol{c_{15}}(\ol{c_{16}}-\ol{c_5})-u_1(u_9-\ol{c_{11}}))$, is a unit in $\aaa_{\FF_p}(\ol{c})$.
\end{enumerate}
\end{proposition}

Both conditions were verified by Gr\"obner basis computations using the Macaulay2 file \texttt{affine\_fiber.m2} of \cite{Git}.

\begin{remark}
\begin{enumerate}
\item Condition (b) shows that $\Delta$ does not vanish at any point of $\Spec\aaa_{\FF_p}(\ol{c})$. By Remark~\ref{rem:genci}, this says that in a neighborhood of this fiber, the commuting ideal is generated by the $12$ polynomials $P_3, P_4,\dots, P_{14}$.
\item The argument below works for any $\ol{c}\in\FF_p^{16}$ satisfying the conditions (a) and (b) above and does not depend upon our specific choice of the affine fiber.
\end{enumerate}
\end{remark}

We now pass from the fiber over $\ol{c}$ to the generic fiber by a flatness argument. Recall that $\Lambda:=\ZZ_{(p)}$ and define
\[T_{\Lambda}:=\Lambda[s_1,s_2,\dots, s_{16}], \qquad \Nn:=(p,s_1-c_1,s_2-c_2,\dots,s_{16}-c_{16})\sub T_{\Lambda}.\]
(Thus, $T=\QQ[s_1,s_2,\dots, s_{16}]=T_{\Lambda}\otimes_{\Lambda}\QQ$.) Note that the ideal $\Nn$ is maximal with residue field $T_{\Lambda}/\Nn=\FF_p$, and $\aaa_{\Lambda}\otimes_{T_{\Lambda}}T_{\Lambda}/\Nn=\aaa_{\FF_p}(\ol{c})$. In the proof of Corollary~\ref{cor:upper}, we also established the following fact:

\begin{lemma} \label{lem:finiteness}
$\aaa_{\Lambda}$ is generated by at most $1825$ elements as a $T_{\Lambda}$-module. In particular, $\aaa_{\Lambda}$ is a finite $T_{\Lambda}$-module.
\end{lemma}

\begin{lemma} \label{lem:flatness}
The localization $\aaa_{\Nn}:=\aaa_{\Lambda}\otimes_{T_{\Lambda}}(T_{\Lambda})_{\Nn}$ is a free $(T_{\Lambda})_{\Nn}$-module of rank $1825$.
\end{lemma}

\begin{proof}
Write $R:=(T_{\Lambda})_{\Nn}$. It is a regular local ring of dimension $17$, with a system of parameters $p,s_1-c_1, s_2-c_2, \dots, s_{16}-c_{16}$. By Lemma~\ref{lem:finiteness}, we have that $\aaa_{\Nn}$ is a finite $R$-module. Fix a maximal ideal $\mathfrak{m}$ of $\aaa_{\Nn}$, which can also be viewed as a point of $\Spec(\aaa_{\Nn}/\Nn\aaa_{\Nn})=\Spec(\aaa_{\FF_p}(\ol{c}))$. We will show that the localization $\aaa_{\mathfrak{m}}$ is a flat $R$-module. Once we can prove that this is true for every maximal ideal $\mathfrak{m}$ of $\aaa_{\Nn}$, we can conclude that the module $\aaa_{\Nn}$ is itself  flat over $R$. A finitely generated flat module over a local ring is free showing that $\aaa_{\Nn}$ is free of rank $\dim_{\FF_p}(\aaa_{\Nn}\otimes_R R/\Nn R)=\dim_{\FF_p}\aaa_{\FF_p}(\ol{c})=1825$, which will complete the proof.

So, we are reduced to showing that $\aaa_{\mathfrak{m}}$ is flat. Let $S:=\Lambda[s_i,u_j:1\leq i\leq 16, 1\leq j\leq 12]$ and let $\mathfrak{m}'$ be the preimage of $\mathfrak{m}$ in $S$, which is a maximal ideal containing $p$. Since $\Delta$ is a unit in $\aaa_{\FF_p}(\ol{c})$, we have that $\Delta\notin\mathfrak{m}'$, and so, $\Delta$ is invertible in $S_{\mathfrak{m}'}$. Therefore, by Corollary~\ref{cor:genci}, we have:
\[\aaa_{\mathfrak{m}}=S_{\mathfrak{m}'}/(P_1,P_2,\dots, P_{14})=S_{\mathfrak{m}'}/(P_3,P_4,\dots, P_{14}).\]
The ring $S_{\mathfrak{m}'}$ is a regular local ring of dimension $29=1+28$: the element $p$ is a non-zero divisor in it, and $S_{\mathfrak{m}'}/pS_{\mathfrak{m}'}$ is the localization of the polynomial ring $\FF_p[s_i,u_j:1\leq i\leq 16, 1\leq j\leq 12]$ at a maximal ideal. Since $\aaa_{\mathfrak{m}}$ is the quotient of $S_{\mathfrak{m}'}$ by an ideal generated by $12$ elements, we have the inequality $\dim(\aaa_{\mathfrak{m}})\geq 29-12=17$. On the other hand, $R\to\aaa_{\mathfrak{m}}$ is a homomorphism of local rings. Therefore, \cite[Theorem 10.10]{Eisenbud} gives that $\dim(\aaa_{\mathfrak{m}})\leq\dim (R)+\dim(\aaa_{\mathfrak{m}}/\Nn\aaa_{\mathfrak{m}})\leq 17$, since $\dim(R)=17$ and $\dim(\aaa_{\mathfrak{m}}/\Nn\aaa_{\mathfrak{m}})=0$, where the latter equality holds since $\aaa_{\mathfrak{m}}/\Nn\aaa_{\mathfrak{m}}$ is a localization of $\aaa_{\Nn}/\Nn\aaa_{\Nn}\simeq \aaa_{\FF_p}(\ol{c})$. Therefore, we have that $\aaa_{\mathfrak{m}}$ is a complete intersection, and thus Cohen-Macaulay. This implies that $\aaa_{\mathfrak{m}}$ is a flat $R$-module by miracle flatness (for example, see \cite[Theorem 18.16]{Eisenbud}).
\end{proof}

\begin{corollary} \label{main:thm}
The algebra $\aaa$ is a free $T$-module of rank $1825$.
\end{corollary}

\begin{proof}
Let $\mathcal{L}:=\mathrm{Frac}(T)=\QQ(s_1,s_2,\dots, s_{16})$, which is also the fraction field of $T_{\Lambda}$ and of $R=(T_{\Lambda})_{\Nn}$. By Lemma~\ref{lem:flatness}, we get:
\[\aaa\otimes_T\mathcal{L}=\aaa_{\Nn}\otimes_R\mathcal{L}\simeq\mathcal{L}^{1825},\]
and so, $\dim_{\mathcal{L}}(\aaa\otimes_T\mathcal{L})=1825$. By Corollary~\ref{cor:upper}, there exists a surjection of $T$-modules $\phi:T^{1825}\twoheadrightarrow\aaa$. The above shows that $\phi\otimes_T\mathcal{L}: \mathcal{L}^{1825}\to \aaa\otimes_T \mathcal{L}$ is a surjection between $\mathcal{L}$-vector spaces of the same dimension $1825$ and, hence, an isomorphism. Therefore, since $T^{1825}$ is torsion-free, we conclude that $\phi$ must itself be an isomorphism, completing the proof.
\end{proof}

\section{Positive characteristic} \label{sec:p>0}

We first fix a prime $p\geq 5$ and a field $K$ of characteristic $p$. (So, we do not necessarily assume that $p=32003$ in this section.) The following criterion can be used for checking Cohen-Macaulayness of $\aaa_K$.

\begin{lemma} \label{lem:p-crit}
Suppose we have the dimension inequality: 
\[\dim_{\FF_p}(\aaa_{\FF_p}/(\ol{s_1}, \ol{s_2},\dots,\ol{s_{16}}))\leq 1825.\]
Then, the ring $\aaa_K$ is Cohen-Macaulay and has Hilbert series $H(t)(1-t)^{-16}$, where $H(t)$ is the same polynomial as in Theorem~\ref{thm:statement}.
\end{lemma}

\begin{proof}
The proof of Cohen-Macaulayness in Corollary~\ref{cor:maincor} can be repeated, as long as we can show that $\aaa_K$ is a free $K[s_1,s_2,\dots,s_{16}]$-module. Note that since $\aaa_K=\aaa_{\FF_p}\otimes_{\FF_p} K$, it suffices to prove the required claim for $K=\FF_p$ which we do below.

Fix a set $\bb_{\FF_p}$ of size at most $1825$ such that $\aaa_{\FF_p}$ is generated by $\bb_{\FF_p}$ as an $\FF_p[s_1,s_2,\dots, s_{16}]$-module. Lift the elements of $\bb_{\FF_p}$ to $\aaa_{\ZZ_{(p)}}$, where $\ZZ_{(p)}$ is the localization of $\ZZ$ at $(p)$. Using a similar argument as the one in the proof of Corollary~\ref{cor:upper} we get that $\aaa_{\ZZ_{(p)}}$ is generated as a $\ZZ_{(p)}[s_1,s_2,\dots, s_{16}]$-module by at most $1825$ elements. We can invert $p$ and use Proposition~\ref{cor:CM-alt} to get that the $\QQ[s_1,s_2,\dots, s_{16}]$-module $ \aaa$ must be freely generated by these elements and the set $\bb_{\FF_p}$ must have size exactly $1825$. Furthermore, since $\ZZ_{(p)}[s_1,s_2,\dots, s_{16}]$ is torsion-free, we get that $\aaa_{\ZZ_{(p)}}$ is a free $\ZZ_{(p)}[s_1,s_2,\dots, s_{16}]$-module of rank $1825$. Going modulo $p$, we get that $\aaa_{\FF_p}$ is a free $\FF_p[s_1,s_2,\dots, s_{16}]$-module of rank $1825$.

The claim regarding the Hilbert series follows from the fact that $\aaa_{\FF_p}/(\ol{s_1}, \ol{s_2},\dots,\ol{s_{16}})$ has Hilbert series $H(t)$ since it has an $\FF_p$-basis given by the set $\bb_{\FF_p}$ which lifts to a $\QQ$-basis for $\aaa/({s_1}, {s_2},\dots,{s_{16}})$.
\end{proof}

Note that, since \[\dim_{\QQ}(\aaa/({s_1}, {s_2},\dots,{s_{16}}))= 1825,\] we can clear denominators and go modulo $p$ for all but finitely many primes. This along with the above lemma implies that:
\begin{corollary}
The ring $\aaa_K$ is Cohen-Macaulay for all but finitely many $p$. Furthermore, $\cc_K$ is reduced whenever it is Cohen-Macaulay.
\end{corollary}

\begin{proof}
The argument for Cohen-Macaulayness is given above. This along with the fact that $\cc_K$ is generically reduced proves reducedness when $K$ is algebraically closed, and thus for general $K$ since we have the embedding $\aaa_K \hookrightarrow \aaa_{\ol K}$.
\end{proof}

The criterion from the above lemma can be checked individually at each prime $p$, but we do not currently have a uniform way of dealing with all primes at once. The Macaulay2 code \texttt{primes\_till\_1000.m2} of \cite{Git} uses the above criterion when $5\leq p\leq 1000$ to prove:

\begin{proposition} \label{prop:5-1000}
Suppose $5\leq p\leq 1000$ and $K$ is a field of characteristic $p$. Then the commuting scheme $\cc_K$ is Cohen-Macaulay and reduced, and has Hilbert series $H(t)(1-t)^{-16}$, where $H(t)$ is the same polynomial as in Theorem~\ref{thm:statement}.
\end{proposition}

Next, we move on to the `bad' characteristic cases.

\subsection{Characteristic $p=3$}
The arguments in the previous section for Cohen-Macaulayness and reducedness of the commuting scheme over positive characteristic fail when $p=3$ for a couple of reasons:
\begin{itemize}
\item The sequence $(\ol{s_1},\ol{s_2},\dots, \ol{s_{16}})$ is no longer a system of parameters modulo $3$. In fact, we have that $\aaa_{\FF_3}/(\ol{s_1},\ol{s_2},\dots, \ol{s_{16}})$ has Krull dimension $1$.
\item The proof of reducedness required generic reducedness of $\cc_K$, which rests on Richardson's irreducibility theorem. This is not available when $p=3$.
\end{itemize}
To remedy the first issue, we replace $s$ by a different sequence of linear forms. Let $L$ be the $16\times 28$ integer matrix with entries in $\{-1,0,1\}$ displayed in Appendix~\ref{sec:app1}, and define
\[(\ell_1,\ell_2,\dots,\ell_{16})^{\mathsf{T}}:=L\cdot(x_1,x_2,\dots,x_{14},y_1,y_2,\dots,y_{14})^{\mathsf{T}}\in\ZZ[x,y]^{16}.\]
(The matrix $L$ was found by random sampling. We did not find a sequence of the form $x_i-y_{\tau(i)}$, for a permutation $\tau$ as in Section~\ref{sec:sop}, that works modulo $3$.) Similarly, we consider the following vector in $\FF_3^{16}$ again obtained by random sampling:
\[c=(-1,1,0,-1,-1,1,0,1,0,0,1,1,1,1,-1,-1).\]

\begin{proposition} \label{prop:m2-char3}
\begin{enumerate}
\item The quotient $\aaa_{\FF_3}/(\ol{\ell_1},\ol{\ell_2},\dots,\ol{\ell_{16}})$ is an $\FF_3$-vector space of dimension $1825$. Furthermore, its Hilbert series is exactly the polynomial $H(t)$ of Theorem~\ref{thm:statement}.
\item The Frobenius map $a\mapsto a^3$ from $\aaa_{\FF_3}/(\ol{\ell_1}-{c_1},\ol{\ell_2}-{c_2},\dots,\ol{\ell_{16}}-{c_{16}})$ to itself is an isomorphism.
\end{enumerate}
\end{proposition}
This proposition follows from a computation using the Macaulay2 file \texttt{char3.m2} of \cite{Git}. Our results in characteristic zero now imply Cohen-Macaulayness in characteristic $3$:

\begin{corollary} \label{cor:CM-char3}
Let $K$ be a field of characteristic $3$. Then, $\aaa_K$ is a free $K[\ol{\ell_1},\ol{\ell_2},\dots,\ol{\ell_{16}}]$-module of rank $1825$. In particular, $\cc_K$ is Cohen-Macaulay and has Hilbert series $H(t)(1-t)^{-16}$, where $H(t)$ is the same polynomial as in Theorem~\ref{thm:statement}.
\end{corollary}

\begin{proof}
Repeating the proof of Corollary~\ref{cor:upper} with $(\ell_i)$ in place of $(s_i)$, we conclude that $(\ell_1, \ell_2,\dots, \ell_{16})$ is a linear system of parameters for $\aaa$ and that $\aaa$ is generated as a $\QQ[\ell_1, \ell_2,\dots, \ell_{16}]$-module by a set $\mathcal{S}$ of $1825$ monomials that span a vector space having Hilbert series $H(t)$. Thus, there is a surjective map of graded $\QQ[\ell_1, \ell_2,\dots, \ell_{16}]$-modules
\[\bigoplus_{f\in\mathcal{S}}\QQ[\ell_1, \ell_2,\dots, \ell_{16}][-\deg(f)] \twoheadrightarrow \aaa,\]
such that both spaces have the same Hilbert series $H(t)(1-t)^{-16}$. Thus, this map must be an isomorphism, showing that $\aaa$ is a free $\QQ[\ell_1, \ell_2,\dots, \ell_{16}]$-module of rank $1825$. Now, we can repeat the proof of Lemma~\ref{lem:p-crit} with $(\ell_i)$ in place of $(s_i)$ to conclude.
\end{proof}

Next, since we do not have generic reducedness in characteristic $3$, we use the following criterion that gives instead a purely computational route to reducedness using Cohen-Macaulayness.

\begin{lemma} \label{lem:reduced-fibre}
Let $\Lambda$ be a domain and let $R$ be the polynomial algebra $\Lambda[x_1,x_2,\dots, x_n]$ for some $n$. Let $A$ be an $R$-algebra which is free of finite rank over $R$. Suppose that there is a maximal ideal $\mathfrak{m}\sub R$ such that the residue field $k:=R/\mathfrak{m}$ is perfect and $A\otimes_R k$ is reduced. Then, $A$ is reduced.
\end{lemma}

\begin{proof}
Choose an $R$-basis $e_1,e_2,\dots, e_N$ of $A$ and consider the discriminant:
\[D:=\det(\Tr_{A/R}(e_ie_j))_{1\leq i,j\leq N}\in R\] 
of the trace form of $A$ over $R$. Then, the image of $D$ in $k$ is the discriminant of the finite $k$-algebra $A\otimes_R k$. A finite algebra over a perfect field is reduced if and only if it is \'etale, which happens if and only if its trace form is non-degenerate. Hence, we conclude that $D\neq 0$. Therefore, the trace form of $A\otimes_R\mathrm{Frac}(R)$ is non-degenerate, so that this algebra is \'etale over $\mathrm{Frac}(R)$, and thus reduced. Since $A$ is free over $R$, it is a subalgebra of $A\otimes_R\mathrm{Frac}(R)$, and must thus be reduced.
\end{proof}

A finite $\FF_p$-algebra is reduced if and only if the Frobenius map $a\mapsto a^p$ is injective. So, the above lemma when $R=\FF_3[\ol{\ell_1}, \ol{\ell_2},\dots, \ol{\ell_{16}}]$,  $A=\aaa_{\FF_3}$ and $\mathfrak{m}=(\ol{\ell_1}-c_1, \ol{\ell_2}-{c_2},\dots, \ol{\ell_{16}}-{c_{16}})$ combined with Proposition~\ref{prop:m2-char3}, Corollary~\ref{cor:CM-char3} and the fact that $\FF_3$ is perfect implies:

\begin{corollary} \label{cor:red-char3}
Let $K$ be a field of characteristic $3$. Then, the scheme $\cc_K$ is reduced.
\end{corollary}

\begin{remark}
Applying Lemma~\ref{lem:reduced-fibre} with $\Lambda=\FF_p$ (for an arbitrary $p\geq 5$), $A=\aaa_{\FF_p}$ and $\mathfrak{m}=(\ol{s_1}-c_1, \ol{s_2}-{c_2},\dots, \ol{s_{16}}-{c_{16}})$, we see that the reducedness of $\cc_{\FF_p}$ follows if we can prove the reducedness of the single finite $\FF_p$-algebra $\aaa_{\FF_p}(\ol{c})$ and the freeness of $A$ as an $R$-module, where the latter can be verified using the computational criterion of Lemma~\ref{lem:p-crit}. This would give another completely computational proof of reducedness that does not depend upon the irreducibility of $\cc_{\FF_p}$.
\end{remark}

\subsection{Characteristic $p=2$}
Let $K$ be a field of characteristic $2$. By \cite[Corollary 4.32]{EK13}, the Lie algebra $\Gg_2$ over $K$ is isomorphic to
$\psl_4$. Therefore, there is an isomorphism of commuting schemes $\cc_K\cong\cc_K(\psl_4)$.

\begin{proposition}\label{prop:char2}
If $\mathrm{char}(K)=2$, the commuting scheme $\cc_K\cong\cc_K(\psl_4)$ is not Cohen-Macaulay. More precisely, there exist irreducible components of dimensions $16$ and $17$ passing through the point given by
$(J,0)\in\cc_K(\mathfrak{psl}_4)$, where $J=E_{2,1}+E_{3,2}+E_{4,3}$.
\end{proposition}

 In the statement above and the proof below, we use $E_{i,j}$ to denote the usual elementary matrix in $\sll_4$, as well as its image under the projection $\sll_4 \twoheadrightarrow\psl_4$. Let $I$ denote the $4\times 4$ identity matrix. The key idea of the proof will be to work with a suitable open subscheme of the commuting scheme that is more amenable to computations. (Such approaches have also appeared in \cite{GG} and \cite{Goy24}, for example.) To that end, we recall that an element $X\in \gl(V)$ is said to be regular if there exists a vector $v\in V$ that is cyclic with respect to $X$. That is, the space $V$ is spanned by the vectors $\{v, Xv, X^2v,\dots, X^{n-1}v\}$. It can be shown that regularity is equivalent to the condition that the dimension of the centralizer $\mathfrak{z}_{\gl_n}(X)$ is minimal (and equal to $n$). The regular locus is a dense open subset of $\gl(V)$.
\begin{proof}[Proof of Proposition~\ref{prop:char2}]
Let $Z\subset\sll_4 \times \sll_4$ be the closed subscheme consisting of pairs $(X,Y)$ where $[X,Y] = \lambda I$ for some $\lambda\in K$. Choose a linear section $\iota:\psl_4 \to \sll_4$ of the projection $\sll_4\twoheadrightarrow \psl_4$. Then, we get an isomorphism:
\[\cc_K(\psl_4)\times \mathbb{A}^2\to Z\]
\[((X,Y),(s,t))\mapsto (\iota(X)+sI, \iota(Y)+tI).\]
Let $\widetilde U\subset Z\times\mathbb A^4$ be the open
subscheme of triples $(X,Y,v)$ such that $v$ is a cyclic vector with respect to $X$. The projection $\widetilde U\to Z$ has open image $U$ consisting of pairs $(X,Y)$ such that $X$ is regular.

Given a regular element $X\in \sll_4$ and a cyclic vector $v$, the matrix $g:=[v|Xv|X^2v|X^3v]$ is invertible. Furthermore, the matrix $g^{-1}Xg$ is equal to the rational canonical form of $X$:
\[g^{-1}Xg=\begin{bmatrix}0&0&0&a\\1&0&0&b\\0&1&0&c\\0&0&1&0\end{bmatrix}\]
for some $a,b,c\in K$. (The bottom right entry of the matrix is zero since $\Tr(X)=0$.) We denote the above matrix by $C_{a,b,c}$ and define the scheme $M$ to be the closed subscheme of $\mathbb{A}^3\times \sll_4$ cut out by the equation $[C_{a,b,c},A]\in K\cdot I$ for $(a,b,c,A)\in \mathbb{A}^3\times \sll_4$. Then, we get another isomorphism:
\[\widetilde{U}\to \mathrm{GL}_4\times M\]
\[(X,Y,v)\mapsto (g, (a, b, c, g^{-1}Yg)).\]
 
Pick a point $(a,b,c,A)\in M$ such that $[C_{a,b,c},A]=\lambda I$ for some $\lambda\in K$ and let $D=E_{1,2}+E_{3,4}$. Then, we have the equality $[C_{a,b,c},D]=I+bE_{1,4}$. Furthermore, we can compute:
\[\Tr(C_{a,b,c})=\Tr(C_{a,b,c}^2)=0, \Tr(C_{a,b,c}^3)=b.\]
Hence,
\[b\lambda=\Tr(C_{a,b,c}^3[C_{a,b,c},A])=\Tr(C_{a,b,c}^4A)-\Tr(C_{a,b,c}^3AC_{a,b,c})=0,\]
and so $[C_{a,b,c},A-\lambda D]=\lambda bE_{1,4}=0$. Now, since the matrix $C_{a,b,c}$ is regular, the centralizer of $C_{a,b,c}$ in $\gl_4$ is spanned by the elements $I,C_{a,b,c},C_{a,b,c}^2$ and $C_{a,b,c}^3$. Therefore, we have
$A=\lambda D+r_0I+r_1C_{a,b,c}+r_2C_{a,b,c}^2+r_3C_{a,b,c}^3$ for some uniquely determined $r_i\in K$ and $0=\Tr(A)=br_3$. Conversely, given any elements $(a,b,c,\lambda, r_0, r_1,r_2,r_3)$ such that $b\lambda=br_3=0$, we can construct $A$ using the formula above and get that $(a,b,c,A)$ lies in $M$. Thus, if we let $B=K[b,\lambda, r_3]/(b\lambda, br_3)$, we get an isomorphism:
\[M\cong\mathbb A^5\times\Spec(B)\]
\[(a,b,c,A)\mapsto ((a,c,r_0,r_1,r_2),(b,\lambda,r_3)).\]
(All of the above steps hold functorially in a $K$-algebra, and so this is an isomorphism of schemes.) The ring $B$ is not Cohen-Macaulay at the origin, since the two minimal prime ideals $(b)$ and $(\lambda, r_3)$ correspond to irreducible components of dimensions $2$ and $1$ respectively, contradicting equidimensionality (see, for example, \cite[Corollary 18.11]{Eisenbud}). We can translate these irreducible components along the maps $M\to \Spec(B)$, $\mathrm{GL}_4\times M\to M$, $\widetilde U\to U$ and
$Z\to\cc_K(\psl_4)$ to get irreducible components of dimensions $17$ and $16$ respectively at the local ring of the point $(J,0)\in \cc_K(\psl_4)$. This shows that $\cc_K$ is not Cohen-Macaulay at the stated point, thus proving the claim.
\end{proof}

\begin{remark}
An analog of the above proposition for the commuting variety of $\mathfrak{pgl}_4$ and, more generally, $\mathfrak{pgl}_n$ was proven in \cite{ROMAN} when the underlying characteristic divides $n$.
\end{remark}

In contrast with the negative conclusion of Proposition~\ref{prop:char2}, the commuting scheme is still reduced in characteristic $2$. For the proof we need an explicit form of the isomorphism $\Gg_2\simeq\psl_4$.

\begin{lemma}\label{lem:g2-psl4}
Let $\mathrm{char}(K)=2$. For $x=\sum_{i=1}^{14}x_ib_i\in\Gg_2$, put
\[
X(x):=\begin{pmatrix}
x_1&x_3&x_5&x_{12}\\
x_9&x_1+x_2&x_4&x_{13}\\
x_{11}&x_{10}&x_2&x_{14}\\
x_6&x_7&x_8&0
\end{pmatrix}\in\sll_4 .
\]
Then $x\mapsto X(x)$ induces an isomorphism of Lie algebras $\varphi\colon\Gg_2\xrightarrow{\ \sim\ }\psl_4$.
\end{lemma}

\begin{proof}
As both spaces have the same dimension and the above map is clearly injective, the lemma follows by showing that the map preserves the Lie bracket, which is a straightforward computation.
\end{proof}

For a $4\times4$ matrix $X$ with $\Tr(X)=0$, write the characteristic polynomial of $X$ as $\det(tI-X)=t^4+e_2(X)t^2+e_3(X)t+e_4(X)$. In characteristic $2$ we have $\det(tI-(X+sI))=\det((t+s)I-X)$, whose $t^2$-coefficient is again $e_2(X)$. So $e_2$ is a well-defined
quadratic function on $\psl_4$, and via Lemma~\ref{lem:g2-psl4} it becomes the following element of $\aaa_K$:
\[
e_2(x)=x_1^2+x_1x_2+x_2^2+x_3x_9+x_4x_{10}+x_5x_{11}+x_6x_{12}+x_7x_{13}+x_8x_{14}.
\]

\begin{lemma}\label{lem:e2-nzd}
The Hilbert series of the ring $\aaa_{\FF_2}$ is $\widetilde{H}(t)(1-t)^{-17}$, where:
\[\widetilde{H}(t)=1+11t+52t^2+134t^3+192t^4+126t^5-10t^6-52t^7-15t^8-t^9.
\]
Consequently, $\dim(\cc_K)=17$ and $\deg(\cc_K)=\widetilde{H}(1)=438$.
Furthermore, the Hilbert series of $\aaa_{\FF_2}/(e_2)$ is $\widetilde{H}(t)(1+t)(1-t)^{-16}$ and $e_2$ is a non-zero divisor on $\aaa_K$.
\end{lemma}

\begin{proof}
Both Hilbert series come from Gr\"obner basis computations using the Macaulay2 file \texttt{char2\_Hilbert.m2} of \cite{Git}. Note that we have the exact sequence:
\[0\to\mathrm{Ann}(e_2)[-2]\to\aaa_{\FF_2}[-2]\xrightarrow{\,e_2\,}\aaa_{\FF_2}\to\aaa_{\FF_2}/e_2\aaa_{\FF_2}\to0,\] 
where $\mathrm{Ann}(e_2)[-2]$ is the annihilator of $e_2$ in $\aaa_{\FF_2}$. Then, computing the graded Euler characteristic of the above sequence gives that the Hilbert series of $\mathrm{Ann}(e_2)[-2]$ is equal to $\mathrm{HS}(\aaa_{\FF_2}/e_2\aaa_{\FF_2})-(1-t^2)\mathrm{HS}(\aaa_{\FF_2})=0$, showing that the annihilator of $e_2$ is zero, and thus it is a non-zero divisor. (The same identity shows that $(1-t^2)\mathrm{HS}(\aaa_{\FF_2})$ is a coefficient-wise lower bound
for $\mathrm{HS}(\aaa_{\FF_2}/e_2\aaa_{\FF_2})$. So it is a valid input for a Hilbert-driven Gr\"obner basis computation of
$\aaa/e_2\aaa$, which is what the code uses.) Since $\aaa_K=\aaa_{\FF_2}\otimes_{\FF_2}K$ and $K$ is flat over $\FF_2$, multiplication by $e_2$ stays injective on $\aaa_K$.
\end{proof}

\begin{remark}
Since some of the coefficients of the polynomial $\widetilde{H}(t)$ are negative, this gives another, independent proof of the fact that $\aaa_{\FF_2}$ is not Cohen-Macaulay. We still chose to include Proposition~\ref{prop:char2} since the geometric argument there is more illuminating and the constructions used there motivated the reducedness proof below.
\end{remark}

\begin{lemma}\label{lem:local-reduced}
Let $K$ be algebraically closed of characteristic $2$, and let $Z\subset\sll_4\times\sll_4$ be as in the proof of Proposition~\ref{prop:char2}. Then $Z$ is reduced at every point $(X,Y)\in Z(K)$ with $e_2(X)\neq0$.
\end{lemma}

\begin{proof}
Write $\det(tI-X)=\prod_{i=1}^m(t-\mu_i)^{n_i}$ with the $\mu_i$ distinct. If $m=1$, then $\det(tI-X)=(t-\mu_1)^4=t^4+\mu_1^4$ which would imply that $e_2(X)=0$. Thus, as $e_2(X)\neq 0$, we can suppose that $m\geq2$.

For a non-increasing tuple $\underline n=(n_1,\dots,n_m)$, let $Z_{\underline n}$ be the closed subscheme of
$\prod_i(\gl_{n_i}\times\gl_{n_i})\times\A^1$ consisting of $((X_i,Y_i)_i,\lambda)$ with $[X_i,Y_i]=\lambda I_{n_i}$
for all $i$ and $\sum_i\Tr(X_i)=\sum_i\Tr(Y_i)=0$. Let $Z^\circ_{\underline n}\subset Z_{\underline n}$ be the open
subscheme where the characteristic polynomials of the $X_i$'s are pairwise coprime. We will consider the morphism:
\[\pi\colon\mathrm{GL}_4\times Z^\circ_{\underline n}\to Z\]
\[(g,(X_i,Y_i)_i,\lambda)\mapsto(g(\oplus_iX_i)g^{-1},\,g(\oplus_iY_i)g^{-1}).\]

Let $\mathrm{Pol}_n$ be the space of monic polynomials of degree $n$ over $K$. The multiplication map
$\prod_i\mathrm{Pol}_{n_i}\to\mathrm{Pol}_4$ is \'etale at $((t-\mu_i)^{n_i})_i$, since these polynomials are
pairwise coprime. Pulling back the map $Z\to\mathrm{Pol}_4$ given by $(X',Y')\mapsto\det(tI-X')$ along the multiplication map gives an \'etale neighborhood $Z'\to Z$ of $(X,Y)$ over which the characteristic polynomial of $X'$ factors as $\prod_i\phi_i$ where $\phi_i$ is monic of degree $n_i$ and they are pairwise coprime. Then, we can decompose $\mathcal O_{Z'}^4=\bigoplus_iV_i$ with $V_i:=\ker(\phi_i(X'))=\ker(\phi_i(X')^2)$ locally free of rank $n_i$.
 
Now, if $[X',Y']=\lambda' I$ for some $\lambda'$, then
$[\phi_i(X')^2,Y']=2\lambda'\phi_i'(X')\phi_i(X')=0$, where $\phi_i'$ is the derivative of $\phi_i$. This shows that $Y'$ preserves $V_i$, and thus, locally on $Z'$, the pair $(X',Y')$ is block diagonal and defines a point of $Z^\circ_{\underline n}$. Let $W:=Z'\times_Z(\mathrm{GL}_4\times Z^\circ_{\underline n})$ be the base change of $\pi$, and let $W'\subseteq W$ be the locus where the characteristic polynomial of the $i$-th block $X_i'$ equals $\phi_i$ for every $i$. Then $W'$ is open and closed and $W'\to Z'$ is a torsor under $\prod_i\mathrm{GL}_{n_i}$ and is thus smooth. Therefore, the composition $W'\to Z'\to Z$ is flat and the map $W'\to\mathrm{GL}_4\times Z^\circ_{\underline n}$ is \'etale, being an open subscheme of the base change of $Z'\to Z$.
 
Thus, by flat descent, the reducedness of $Z$ at $(X,Y)$ is implied by the reducedness of $Z^\circ_{\underline n}$ for each $\underline n$. The various possibilities for $\underline n$ are $(1,1,1,1)$, $(2,1,1)$, $(3,1)$ and
$(2,2)$. If some $n_i$ is odd, then $\Tr([X_i,Y_i])=0$ forces $\lambda=0$. Hence, we have
\[
Z_{(1,1,1,1)}\cong\A^6,\qquad Z_{(2,1,1)}\cong\cc(\gl_2)\times\A^2,\qquad Z_{(3,1)}\cong\cc(\gl_3),
\]
whereas $Z_{(2,2)}$ is the scheme of $(X_1,Y_1,X_2,Y_2,\lambda)\in\gl_2^4\times\A^1$ with
$[X_1,Y_1]=[X_2,Y_2]=\lambda I_2$ and $\Tr(X_1+X_2)=\Tr(Y_1+Y_2)=0$. For $\cc(\gl_2)$, $\cc(\gl_3)$ and
$Z_{(2,2)}$ over $\FF_2$, we use the Macaulay2 file \texttt{char2\_reduced.m2} of \cite{Git} to show that the defining ideal is radical. Therefore, these schemes are reduced over $\FF_2$, and hence over $K$ since $\FF_2$ is perfect.
\end{proof}

\begin{proposition}\label{prop:red-char2}
If $\mathrm{char}(K)=2$, the commuting scheme $\cc_K$ is reduced. It has dimension $17$, degree $438$ and Hilbert series $\widetilde{H}(t)(1-t)^{-17}$, where $\widetilde{H}(t)$ is the polynomial stated in Theorem~\ref{thm:char2,3}.
\end{proposition}

\begin{proof}
By Lemma~\ref{lem:e2-nzd}, we have injections
$\aaa_K\hookrightarrow(\aaa_K)_{e_2}\hookrightarrow(\aaa_{\ol K})_{e_2}$. So it suffices to show that the open
subscheme $\{e_2\neq0\}$ of $\cc_{\ol K}$ is reduced. Under the isomorphism
$\cc_{\ol K}(\psl_4)\times\A^2\cong Z$ from the proof of Proposition~\ref{prop:char2}, the function $e_2$ on the left
corresponds to $e_2(X)$ on $Z$. A scheme of finite type over
$\ol K$ is reduced if it is reduced at all of its $\ol K$-points. So the claim follows from
Lemma~\ref{lem:local-reduced}. The dimension, degree and Hilbert series were computed in Lemma~\ref{lem:e2-nzd}.
\end{proof}

\appendix

\section{The structure constants of $\Gg_2$} \label{sec:app1}

Recall the following basis that was fixed for the Lie algebra $\Gg_2$:
\[(b_i)_{1\leq i\leq 14}:=(h_{\alpha}, h_{\beta}, e_\alpha,e_\beta,e_{\alpha+\beta},e_{2\alpha+\beta},e_{3\alpha+\beta},e_{3\alpha+2\beta},f_\alpha,f_\beta,f_{\alpha+\beta},f_{2\alpha+\beta},
 f_{3\alpha+\beta},f_{3\alpha+2\beta}).\]
 With respect to this basis, we can write down the commutator matrix (displayed on the next page). We use the convention that the $(i,j)^{th}$ entry is the commutator $[b_i,b_j]$.

 \begin{landscape}
 \thispagestyle{empty}
\begin{table}[ht] 
\centering
\setlength{\arraycolsep}{3pt}
\caption{Commutator table of $\Gg_2$}
\resizebox{\linewidth}{!}{$
\begin{array}{c|cc|cccccc|cccccc}
 & h_{\alpha} & h_{\beta} & e_{\alpha} & e_{\beta} & e_{\alpha+\beta} & e_{2\alpha+\beta} & e_{3\alpha+\beta} & e_{3\alpha+2\beta} & f_{\alpha} & f_{\beta} & f_{\alpha+\beta} & f_{2\alpha+\beta} & f_{3\alpha+\beta} & f_{3\alpha+2\beta} \\ \hline
h_{\alpha} & 0 & 0 & 2e_{\alpha} & -3e_{\beta} & -e_{\alpha+\beta} & e_{2\alpha+\beta} & 3e_{3\alpha+\beta} & 0 & -2f_{\alpha} & 3f_{\beta} & f_{\alpha+\beta} & -f_{2\alpha+\beta} & -3f_{3\alpha+\beta} & 0 \\
h_{\beta} & 0 & 0 & -e_{\alpha} & 2e_{\beta} & e_{\alpha+\beta} & 0 & -e_{3\alpha+\beta} & e_{3\alpha+2\beta} & f_{\alpha} & -2f_{\beta} & -f_{\alpha+\beta} & 0 & f_{3\alpha+\beta} & -f_{3\alpha+2\beta} \\ \hline
e_{\alpha} & -2e_{\alpha} & e_{\alpha} & 0 & e_{\alpha+\beta} & 2e_{2\alpha+\beta} & 3e_{3\alpha+\beta} & 0 & 0 & h_{\alpha} & 0 & 3f_{\beta} & 2f_{\alpha+\beta} & f_{2\alpha+\beta} & 0 \\
e_{\beta} & 3e_{\beta} & -2e_{\beta} & -e_{\alpha+\beta} & 0 & 0 & 0 & e_{3\alpha+2\beta} & 0 & 0 & h_{\beta} & -f_{\alpha} & 0 & 0 & f_{3\alpha+\beta} \\
e_{\alpha+\beta} & e_{\alpha+\beta} & -e_{\alpha+\beta} & -2e_{2\alpha+\beta} & 0 & 0 & -3e_{3\alpha+2\beta} & 0 & 0 & -3e_{\beta} & e_{\alpha} & -h_{\alpha} -3h_{\beta}& 2f_{\alpha} & 0 & f_{2\alpha+\beta} \\
e_{2\alpha+\beta} & -e_{2\alpha+\beta} & 0 & -3e_{3\alpha+\beta} & 0 & 3e_{3\alpha+2\beta} & 0 & 0 & 0 & -2e_{\alpha+\beta} & 0 & -2e_{\alpha} & 2h_{\alpha}+3h_{\beta} & -f_{\alpha} & f_{\alpha+\beta} \\
e_{3\alpha+\beta} & -3e_{3\alpha+\beta} & e_{3\alpha+\beta} & 0 & -e_{3\alpha+2\beta} & 0 & 0 & 0 & 0 & -e_{2\alpha+\beta} & 0 & 0 & e_{\alpha} & -h_{\alpha}-h_{\beta} & f_{\beta} \\
e_{3\alpha+2\beta} & 0 & -e_{3\alpha+2\beta} & 0 & 0 & 0 & 0 & 0 & 0 & 0 & -e_{3\alpha+\beta} & -e_{2\alpha+\beta} & -e_{\alpha+\beta} & -e_{\beta} & h_{\alpha}+2h_{\beta} \\ \hline
f_{\alpha} & 2f_{\alpha} & -f_{\alpha} & -h_{\alpha} & 0 & 3e_{\beta} & 2e_{\alpha+\beta} & e_{2\alpha+\beta} & 0 & 0 & f_{\alpha+\beta} & 2f_{2\alpha+\beta} & 3f_{3\alpha+\beta} & 0 & 0 \\
f_{\beta} & -3f_{\beta} & 2f_{\beta} & 0 & -h_{\beta} & -e_{\alpha} & 0 & 0 & e_{3\alpha+\beta} & -f_{\alpha+\beta} & 0 & 0 & 0 & f_{3\alpha+2\beta} & 0 \\
f_{\alpha+\beta} & -f_{\alpha+\beta} & f_{\alpha+\beta} & -3f_{\beta} & f_{\alpha} & h_{\alpha}+3h_{\beta} & 2e_{\alpha} & 0 & e_{2\alpha+\beta} & -2f_{2\alpha+\beta} & 0 & 0 & -3f_{3\alpha+2\beta} & 0 & 0 \\
f_{2\alpha+\beta} & f_{2\alpha+\beta} & 0 & -2f_{\alpha+\beta} & 0 & -2f_{\alpha} & -2h_{\alpha}-3h_{\beta} & -e_{\alpha} & e_{\alpha+\beta} & -3f_{3\alpha+\beta} & 0 & 3f_{3\alpha+2\beta} & 0 & 0 & 0 \\
f_{3\alpha+\beta} & 3f_{3\alpha+\beta} & -f_{3\alpha+\beta} & -f_{2\alpha+\beta} & 0 & 0 & f_{\alpha} & h_{\alpha}+h_{\beta} & e_{\beta} & 0 & -f_{3\alpha+2\beta} & 0 & 0 & 0 & 0 \\
f_{3\alpha+2\beta} & 0 & f_{3\alpha+2\beta} & 0 & -f_{3\alpha+\beta} & -f_{2\alpha+\beta} & -f_{\alpha+\beta} & -f_{\beta} & -h_{\alpha}-2h_{\beta} & 0 & 0 & 0 & 0 & 0 & 0
\end{array}
$}
\end{table}
\end{landscape}

 Using this table, we can compute the equations defining the commuting ideal: Let $x=\sum_{i=1}^{14}x_ib_i$ and $y=\sum_{i=1}^{14}y_ib_i$. Then, $[x,y]=\sum_{k=1}^{14}Q_k(x,y) b_k$, where:
 \begin{align*}
Q_1(x,y)&=w_{3,9}-w_{5,11}+2w_{6,12}-w_{7,13}+w_{8,14},\\
Q_2(x,y)&=w_{4,10}-3w_{5,11}+3w_{6,12}-w_{7,13}+2w_{8,14},\\
Q_3(x,y)&=2w_{1,3}-w_{2,3}+w_{5,10}-2w_{6,11}+w_{7,12},\\
Q_4(x,y)&=-3w_{1,4}+2w_{2,4}-3w_{5,9}-w_{8,13},\\
Q_5(x,y)&=-w_{1,5}+w_{2,5}+w_{3,4}-2w_{6,9}-w_{8,12},\\
Q_6(x,y)&=w_{1,6}+2w_{3,5}-w_{7,9}-w_{8,11},\\
Q_7(x,y)&=3w_{1,7}-w_{2,7}+3w_{3,6}-w_{8,10},\\
Q_8(x,y)&=w_{2,8}+w_{4,7}-3w_{5,6},\\
Q_9(x,y)&=-2w_{1,9}+w_{2,9}-w_{4,11}+2w_{5,12}-w_{6,13},\\
Q_{10}(x,y)&=3w_{1,10}-2w_{2,10}+3w_{3,11}+w_{7,14},\\
Q_{11}(x,y)&=w_{1,11}-w_{2,11}+2w_{3,12}+w_{6,14}+w_{9,10},\\
Q_{12}(x,y)&=-w_{1,12}+w_{3,13}+w_{5,14}+2w_{9,11},\\
Q_{13}(x,y)&=-3w_{1,13}+w_{2,13}+w_{4,14}+3w_{9,12},\\
Q_{14}(x,y)&=-w_{2,14}+w_{10,13}-3w_{11,12}.
 \end{align*}
(Here, we have used the notation $w_{i,j}:=x_iy_j-x_jy_i$.)
\begin{remark} \label{rem:perm}
The permutation $\sigma$ defined at the beginning of Section~\ref{sec:sop} was chosen so that for each $1\leq i\leq 14$, the commutator $[b_i,b_{\sigma(i)}]\neq 0$. In fact, each $[b_i,b_{\sigma(i)}]$ is a non-zero scalar multiple of a distinct $b_j$. Therefore, the permutation $\sigma$ could be thought of as a Lie theoretic analog of a complete mapping (see, for example, \cite{Hall}).

This permutation was found by playing around with the above commutator matrix, attempting to select $14$ entries so that there is one in each row and each column. The above property of the permutation ensures that when we substitute $x_i=y_{\sigma(i)}$ in the $14$ quadratic equations $Q_k(x,y)$ defining the scheme $\cc$, each of them ends up with exactly one square term. For example, we observe that:
\[(Q_1)_{y_i\mapsto x_{\sigma^{-1}(i)}}=x_3x_{12}-x_9^2 - x_5x_8+x_{11}x_{14} + 2x_6x_7-2x_3x_{12}-x_6x_7 +x_4x_{13}+x_2x_8-x_{13}x_{14},\]
etc. Heuristically, we expect this to be one of the reasons why the above permutation led to a faster Gr\"obner basis computation compared to some other options we tried.
\end{remark}

Using the commutator matrix, we can also compute the element $\Delta$ in Corollary~\ref{cor:genci} in our coordinates. In the notation there:
\begin{align*}
\Delta&=\lambda_1\mu_2-\lambda_2\mu_1\\
&=\sum\limits_{k,l=1}^{14}\kappa(b_k,b_1)\kappa(b_l,b_2)(x_ky_l-x_ly_k)\\
&=\sum\limits_{k,l=1}^{2}\kappa(b_k,b_1)\kappa(b_l,b_2)(x_ky_l-x_ly_k)\\
&=(\kappa(b_1,b_1)\kappa(b_2,b_2)-\kappa(b_1,b_2)^2)(x_1y_2-x_2y_1).
\end{align*}
Here, we have used the fact that $\kappa(b_k,b_l)=0$ whenever $k\leq 2$ and $l\geq 3$. This equality holds because $\ad(b_k)$ necessarily preserves each root space (since $b_k$ lies in the Cartan subalgebra) and $\ad(b_l)$ necessarily moves each root space to a different one.

Then, using the commutator table, we can compute that:
\[\kappa(b_1,b_1)=\Tr(\ad(b_1)^2)=48, \quad\kappa(b_2,b_2)=\Tr(\ad(b_2)^2)=16, \quad \kappa(b_1,b_2)=\Tr(\ad(b_1)\ad(b_2))=-24,\]
and so $\Delta=192(x_1y_2-x_2y_1)$. In terms of the generators $s_i,u_j$ defined in Section~\ref{sec:sop}, we can write $\Delta=192(s_{15}(s_{16}-s_5)-u_1(u_9-s_{11}))$.

Finally, we note down the $16\times 28$ matrix $L$ that is used to define a system of parameters in characteristic $3$:
\[L=\left(\begin{smallmatrix}
  0 &  0 &  1 &  1 &  0 &  0 &  1 &  0 &  0 & -1 &  0 &  0 &  0 &  0 & -1 &  0 &  0 &  0 & -1 &  0 &  1 &  0 &  1 &  1 & -1 &  0 &  0 & -1 \\
  0 &  0 &  0 &  0 &  0 & -1 & -1 &  0 &  1 & -1 &  0 &  0 & -1 &  0 & -1 &  1 &  0 &  0 &  0 &  0 &  0 &  0 & -1 &  0 &  0 &  0 &  0 &  0 \\
  0 &  1 & -1 &  0 &  0 & -1 &  0 &  0 &  1 &  1 &  1 &  0 &  0 &  0 & -1 &  1 &  0 &  1 & -1 &  1 &  1 &  1 &  1 &  0 &  1 &  0 &  1 &  0 \\
  1 &  1 & -1 & -1 &  0 &  0 & -1 &  0 &  0 &  0 &  1 & -1 &  0 &  1 &  0 &  0 & -1 & -1 & -1 &  0 &  0 &  0 &  0 &  1 &  1 &  1 &  0 &  0 \\
  0 &  0 & -1 &  1 &  0 &  0 &  0 &  0 & -1 &  1 & -1 &  1 &  0 &  1 &  1 &  0 &  1 &  0 & -1 &  0 &  1 &  0 &  1 &  0 & -1 &  0 &  0 &  0 \\
  0 &  1 &  0 &  0 &  0 & -1 & -1 & -1 &  1 &  0 & -1 &  1 &  1 &  1 & -1 & -1 & -1 & -1 &  1 &  1 & -1 &  0 &  0 &  0 & -1 &  0 &  0 & -1 \\
  0 & -1 & -1 &  0 &  0 & -1 &  0 &  1 & -1 &  0 &  0 &  0 &  1 &  1 &  0 & -1 &  0 &  0 &  0 &  0 &  0 & -1 & -1 & -1 & -1 &  1 &  0 &  1 \\
  1 & -1 &  0 &  0 & -1 &  0 &  0 &  1 &  1 & -1 & -1 &  0 & -1 &  0 &  0 &  1 &  0 & -1 &  0 &  0 & -1 &  0 &  0 &  1 &  0 &  1 &  0 &  0 \\
  0 &  0 &  1 &  1 &  1 & -1 &  1 &  1 &  0 &  0 & -1 &  0 &  0 & -1 &  0 &  0 & -1 &  1 &  0 & -1 &  1 & -1 & -1 &  0 & -1 &  0 &  0 & -1 \\
  0 & -1 &  1 & -1 & -1 & -1 &  0 &  1 &  0 &  1 &  0 &  0 &  1 &  0 &  0 &  0 &  1 & -1 &  0 & -1 & -1 & -1 & -1 &  1 &  0 &  1 &  0 &  1 \\
  0 & -1 &  1 &  0 & -1 & -1 & -1 &  0 &  0 &  0 &  0 &  0 &  0 &  1 & -1 & -1 &  0 & -1 &  1 & -1 &  1 &  1 & -1 & -1 & -1 &  0 &  1 &  1 \\
  0 &  1 &  0 &  1 & -1 &  1 &  1 &  0 & -1 &  0 &  0 &  0 &  1 & -1 &  1 & -1 &  1 & -1 & -1 &  0 &  0 & -1 &  0 &  0 & -1 & -1 &  1 &  1 \\
  1 & -1 &  1 & -1 &  1 &  1 & -1 & -1 & -1 &  1 &  1 &  1 &  0 &  0 &  0 &  0 & -1 &  1 & -1 &  1 &  0 &  0 & -1 &  0 &  0 &  0 &  0 &  1 \\
  0 &  1 &  1 & -1 &  1 & -1 &  0 &  0 & -1 &  0 & -1 &  1 &  0 &  0 &  1 & -1 &  0 &  0 &  0 &  0 &  0 &  1 &  0 &  0 &  1 &  0 &  1 &  0 \\
  1 &  0 & -1 &  0 &  0 & -1 &  0 &  1 & -1 &  0 & -1 &  0 & -1 &  0 & -1 & -1 &  0 &  0 &  0 &  1 &  0 &  0 & -1 &  0 & -1 &  0 &  0 & -1 \\
  0 &  1 & -1 &  1 &  0 &  1 &  0 &  0 &  1 & -1 & -1 &  0 &  0 &  0 & -1 &  0 &  1 &  1 &  0 &  1 &  0 & -1 &  0 &  1 & -1 &  0 &  0 & -1
\end{smallmatrix}\right).\]

\section{Degree table} \label{sec:degree-table}

By Remark~\ref{rem:degree-hypo}, we can use Theorem~\ref{thm:intro-degree-formula} to compute the degree of the commuting scheme $\cc(\Gg)$ whenever $\Gg$ is of type $A_l$ ($l\leq 15$), $B_l$ ($l\leq6$), $C_l$ ($l\leq6$), $D_l$ ($l\leq7$), $E_6$ or $G_2$. The following table of values was obtained by running the Python code \texttt{degree\_comp.py} in \cite{Git}.

\begingroup
\renewcommand{\arraystretch}{1.25}
\setlength{\tabcolsep}{5pt}
\begin{longtable}{llrrr>{\raggedright\arraybackslash}p{4.9cm}}
\caption{Degrees of the commuting scheme $\cc(\Gg)$}
\label{tab:degrees}\\
\toprule
$\Gg$ & type & $\dim(\mathfrak g)$ & $N$ & $|W|$ & $\deg(\cc(\Gg))$ \\
\midrule
\endfirsthead
\multicolumn{6}{l}{\small\emph{Table~\ref{tab:degrees} (continued)}}\\[2pt]
\toprule
$\Gg$ & type & $\dim(\Gg)$ & $N$ & $|W|$ & $\deg(\cc(\Gg))$ \\
\midrule
\endhead
\midrule
\multicolumn{6}{r}{\small\emph{}}\\
\endfoot
\bottomrule
\endlastfoot
$\mathfrak{sl}_{2}$  & $A_1$ & 3   & 1  & 2        & 3 \\
$\mathfrak{sl}_{3}$  & $A_2$ & 8   & 3  & 6        & 31 \\
$\mathfrak{sl}_{4}$  & $A_3$ & 15  & 6  & 24       & 1145 \\
$\mathfrak{sl}_{5}$  & $A_4$ & 24  & 10 & 120      & 154881 \\
$\mathfrak{sl}_{6}$  & $A_5$ & 35  & 15 & 720      & 77899563 \\
$\mathfrak{sl}_{7}$  & $A_6$ & 48  & 21 & 5040     & 147226330175 \\
$\mathfrak{sl}_{8}$  & $A_7$ & 63  & 28 & 40320    & 1053765855157617 \\
$\mathfrak{sl}_{9}$  & $A_8$ & 80  & 36 & 362880   & 28736455088578690945 \\
$\mathfrak{sl}_{10}$ & $A_9$ & 99  & 45 & 3628800  & 3000127124463666294963283 \\
$\mathfrak{sl}_{11}$ & $A_{10}$ & 120 & 55  & 39916800       & \seqsplit{1203831304687539089648950490463} \\
$\mathfrak{sl}_{12}$ & $A_{11}$ & 143 & 66  & 479001600      & \seqsplit{1862632561783036151478238040096092649} \\
$\mathfrak{sl}_{13}$ & $A_{12}$ & 168 & 78  & 6227020800     & \seqsplit{11143500837236042423379349834982088594105985} \\
$\mathfrak{sl}_{14}$ & $A_{13}$ & 195 & 91  & 87178291200    & \seqsplit{258389645665477633207631198815389851306350911123579} \\
$\mathfrak{sl}_{15}$ & $A_{14}$ & 224 & 105 & 1307674368000  & \seqsplit{23268964120590079117360931838307760546960300665351208275071} \\
$\mathfrak{sl}_{16}$ & $A_{15}$ & 255 & 120 & 20922789888000 & \seqsplit{8152788880952641347488179079698833772730621821001288826319965501665} \\
\midrule
$\mathfrak{so}_{5}$  & $B_2$ & 10  & 4  & 8        & 117 \\
$\mathfrak{so}_{7}$  & $B_3$ & 21  & 9  & 48       & 59815 \\
$\mathfrak{so}_{9}$  & $B_4$ & 36  & 16 & 384      & 432811017 \\
$\mathfrak{so}_{11}$ & $B_5$ & 55  & 25 & 3840     & 45666533329515 \\
$\mathfrak{so}_{13}$ & $B_6$ & 78  & 36 & 46080    & 71329953700904336445 \\
\midrule
$\mathfrak{sp}_{6}$  & $C_3$ & 21  & 9  & 48       & 59367 \\
$\mathfrak{sp}_{8}$  & $C_4$ & 36  & 16 & 384      & 429183945 \\
$\mathfrak{sp}_{10}$ & $C_5$ & 55  & 25 & 3840     & 45242989425963 \\
$\mathfrak{sp}_{12}$ & $C_6$ & 78  & 36 & 46080    & 70589466680235089789 \\
\midrule
$\mathfrak{so}_{8}$  & $D_4$ & 28  & 12 & 192      & 2084305 \\
$\mathfrak{so}_{10}$ & $D_5$ & 45  & 20 & 1920     & 54954085041 \\
$\mathfrak{so}_{12}$ & $D_6$ & 66  & 30 & 23040    & 21417377177253593 \\
$\mathfrak{so}_{14}$ & $D_7$ & 91  & 42 & 322560   & 124994013903952572250921 \\
\midrule
$\mathfrak{e}_{6}$   & $E_6$ & 78  & 36 & 51840    & 64322574610840806385 \\
\midrule
$\Gg_{2}$   & $G_2$ & 14  & 6  & 12       & 1825 \\
\end{longtable}
\endgroup

\begin{remark} \label{rem:crt}
The code does not evaluate the right hand side of Theorem~\ref{thm:intro-degree-formula} directly over $\QQ$, but modulo each of the seven (ten, in type $A$) largest primes below $2^{30}$, and reconstructs the answer from these residues by the Chinese remainder theorem. This yields the correct value of $\deg(\cc(\Gg))$ because of the following bound: The scheme $\cc(\Gg)$ is cut out from $\Gg\times\Gg$ by $\dim(\Gg)$ quadratic polynomials. Since $\cc(\Gg)$ is irreducible and has multiplicity one along its unique irreducible
component (by generic reducedness), Equation~\eqref{eq:associativity} says that $\deg(\cc(\Gg))$ equals the degree of the underlying variety, which is bounded above by $2^{\dim(\Gg)}$ by the refined
B\'ezout theorem (see, for example, \cite[Example~8.4.6]{Ful98}). The product of the chosen primes exceeds this bound for every entry of Table~\ref{tab:degrees}, and so $\deg(\cc(\Gg))$ is determined by its residues.
\end{remark}

\section{Macaulay2 scripts} \label{sec:app-m2}

The following table lists the various Macaulay2 codes used in this paper that can be found in the repository \cite{Git}.

\begin{center}
\begin{tabular}{|c|c|c|c|c|}
\hline
Name of the Macaulay2 file&Statement verified by the code&Approximate run time\\
\hline
\texttt{system\_of\_parameters.m2}& Proposition~\ref{prop:m2-first}&$5$ seconds\\
\hline
\texttt{affine\_fiber.m2}& Proposition~\ref{prop:m2-second}&$25$ seconds\\
\hline
\texttt{primes\_till\_1000.m2}& Proposition~\ref{prop:5-1000}&$550$ seconds\\
\hline
\texttt{char3.m2}& Proposition~\ref{prop:m2-char3}&$350$ seconds\\
\hline
\texttt{char2\_Hilbert.m2}& Lemma~\ref{lem:e2-nzd}&$360$ seconds\\
\hline
\texttt{char2\_reduced.m2}& Lemma~\ref{lem:local-reduced}&$1$ second\\
\hline
\end{tabular}
\end{center}

\printbibliography
\end{document}